\documentclass[12pt,reqno]{amsart}
\usepackage{latexsym}
\usepackage{amssymb}
\usepackage{mathrsfs}
\usepackage{amsmath}
\usepackage{comment}
\usepackage{fancybox,color}
\usepackage{enumerate}
\usepackage[latin1]{inputenc}
\usepackage{soul} 
\usepackage[colorlinks=true, linkcolor=blue, citecolor=blue]{hyperref}

{\catcode`p =12 \catcode`t =12 \gdef\eeaa#1pt{#1}}      %
\def\accentadjtext#1{\setbox0\hbox{$#1$}\kern   %
                \expandafter\eeaa\the\fontdimen1\textfont1 \ht0 }
\def\accentadjscript#1{\setbox0\hbox{$#1$}\kern %
                \expandafter\eeaa\the\fontdimen1\scriptfont1 \ht0 }
\def\accentadjscriptscript#1{\setbox0\hbox{$#1$}\kern   %
                \expandafter\eeaa\the\fontdimen1\scriptscriptfont1 \ht0 }
\def\accentadjtextback#1{\setbox0\hbox{$#1$}\kern       %
                -\expandafter\eeaa\the\fontdimen1\textfont1 \ht0 }
\def\accentadjscriptback#1{\setbox0\hbox{$#1$}\kern     %
                -\expandafter\eeaa\the\fontdimen1\scriptfont1 \ht0 }
\def\accentadjscriptscriptback#1{\setbox0\hbox{$#1$}\kern %
                -\expandafter\eeaa\the\fontdimen1\scriptscriptfont1 \ht0 }
\def\itoverline#1{{\mathsurround0pt\mathchoice
        {\rlap{$\accentadjtext{\displaystyle #1}
                \accentadjtext{\vrule height1.593pt}
                \overline{\phantom{\displaystyle #1}
                \accentadjtextback{\displaystyle #1}}$}{#1}}
        {\rlap{$\accentadjtext{\textstyle #1}
                \accentadjtext{\vrule height1.593pt}
                \overline{\phantom{\textstyle #1}
                \accentadjtextback{\textstyle #1}}$}{#1}}
        {\rlap{$\accentadjscript{\scriptstyle #1}
                \accentadjscript{\vrule height1.593pt}
                \overline{\phantom{\scriptstyle #1}
                \accentadjscriptback{\scriptstyle #1}}$}{#1}}
        {\rlap{$\accentadjscriptscript{\scriptscriptstyle #1}
                \accentadjscriptscript{\vrule height1.593pt}
                \overline{\phantom{\scriptscriptstyle #1}
                \accentadjscriptscriptback{\scriptscriptstyle #1}}$}{#1}}}}

\newcommand{\iol}{\itoverline}

\newcommand{\ch}[1]{{\mbox{\raise 1pt\hbox{\large$\chi$}}}_{\lower 1pt\hbox{$\scriptstyle #1$}}}

\usepackage[left=2.1cm,top=2.3cm,right=2.1cm]{geometry}

\def\1{\raisebox{2pt}{\rm{$\chi$}}}

\newtheorem{theorem}{Theorem}[section]

\newtheorem{lemma}[theorem]{Lemma}

\theoremstyle{definition}
\newtheorem{definition}[theorem]{Definition}
\newtheorem{remark}[theorem]{Remark}
\newtheorem{example}[theorem]{Example}

\newcommand{\R}{{\mathbb R}}

\newcommand{\N}{{\mathbb N}}

\newcommand{\D}{{\mathcal D}}

\newcommand\diam{\operatorname{diam}}

\def\1{\raisebox{2pt}{\rm{$\chi$}}}

\def\vint_#1{\mathchoice%
        {\mathop{\kern 0.2em\vrule width 0.6em height 0.69678ex depth -0.58065ex
                \kern -0.8em \intop}\nolimits_{\kern -0.4em#1}}%
        {\mathop{\kern 0.1em\vrule width 0.5em height 0.69678ex depth -0.60387ex
                \kern -0.6em \intop}\nolimits_{#1}}%
        {\mathop{\kern 0.1em\vrule width 0.5em height 0.69678ex
            depth -0.60387ex
                \kern -0.6em \intop}\nolimits_{#1}}%
        {\mathop{\kern 0.1em\vrule width 0.5em height 0.69678ex depth -0.60387ex
                \kern -0.6em \intop}\nolimits_{#1}}}
\def\vintslides_#1{\mathchoice%
        {\mathop{\kern 0.1em\vrule width 0.5em height 0.697ex depth -0.581ex
                \kern -0.6em \intop}\nolimits_{\kern -0.4em#1}}%
        {\mathop{\kern 0.1em\vrule width 0.3em height 0.697ex depth -0.604ex
                \kern -0.4em \intop}\nolimits_{#1}}%
        {\mathop{\kern 0.1em\vrule width 0.3em height 0.697ex depth -0.604ex
                \kern -0.4em \intop}\nolimits_{#1}}%
        {\mathop{\kern 0.1em\vrule width 0.3em height 0.697ex depth -0.604ex
                \kern -0.4em \intop}\nolimits_{#1}}}

\newcommand{\intav}{\vint}

\title[Limits at Infinity]{Limits  at infinity for  Haj{\l}asz--Sobolev functions in  metric spaces}

\author[A. Agarwal]{Angha Agarwal}   %
\address[A.A.]{University of Jyvaskyla, Department of Mathematics and Statistics, P.O. Box 35, FI-40014 University of Jyvaskyla, Finland}
\email{angha.a.agarwal@jyu.fi}

\author[A.V.\! V\"ah\"akangas]{Antti V. V\"ah\"akangas}
\address[A.V.V.]{University of Jyvaskyla, Department of Mathematics and Statistics, P.O. Box 35, FI-40014 University of Jyvaskyla, Finland} 
\email{antti.vahakangas@iki.fi}

\makeatletter

\@namedef{subjclassname@2020}{{\mdseries 2020} Mathematics Subject Classification}
\makeatother


\thanks{The authors are  grateful to Pekka Koskela for 
suggesting the topic   of the paper   and for fruitful discussions. A. A. was supported by the Research Council of Finland Centre of Excellence in Randomness and Structures (Project number 364210).}

\begin{document}

\begin{abstract}
We study limits at infinity for homogeneous Haj{\l}asz--Sobolev functions defined on uniformly perfect metric spaces equipped with a doubling measure. We prove that a quasicontinuous representative of such a function has a pointwise limit at infinity outside an exceptional set, defined in terms of a variational relative capacity. 
Our framework refines earlier approaches that relied on Hausdorff content rather than relative capacity, and it extends previous results for homogeneous Newtonian and fractional Sobolev functions.
\end{abstract}

\maketitle

\section{Introduction}
In this paper, we study limits at infinity for functions in the homogeneous Haj\l asz--Sobolev space within the framework of metric measure spaces. 
Related problems have been studied for functions in the first-order Sobolev space in \cite{ErKoNg,Fefferman,KlKoNg,KoNg,Portnov,Upenskii}. Similar questions have been addressed for functions of bounded variation in \cite{PaKh}, and for functions in fractional Sobolev spaces in \cite{AKM}.

  We briefly review the results in \cite{KlKoNg,AKM} that are most relevant to our work. Kline, Koskela, and Nguyen \cite[Theorem~1.1]{KlKoNg} obtained limits at infinity for functions in the first-order homogeneous Newton--Sobolev space outside thin exceptional sets, which have sufficiently small Hausdorff contents in dyadic annuli at infinity. Their results rely on the assumption that the ambient metric measure space is
complete, equipped with a $d$-regular measure, and supports a suitable Poincar\'e inequality. In a recent work, Agarwal, Koskela, and Mohanta \cite{AKM} established the existence of limits at infinity for  functions in homogeneous fractional Sobolev spaces, outside thin exceptional sets characterized by small Hausdorff content at infinity. In their setting, the metric space is $\R^n$ equipped with the Euclidean metric and Lebesgue measure.

Extending the approaches in \cite{KlKoNg} and \cite{AKM}, we demonstrate the existence of limits at infinity for functions in the homogeneous Haj\l asz--Sobolev space $\dot{M}^{s,p}(X,d,\mu)$, where  $(X,d,\mu)$ is a metric measure space, $0<s\leq1$ and $0<p<\infty$. 
A priori, we do not assume  that the underlying metric space $X$ is connected. Instead, it is assumed to be uniformly perfect and equipped with a doubling measure $\mu$,  providing the necessary geometric and measure-theoretic part of our framework described in Section \ref{e.preliminaries}.
The choice of the homogeneous Haj\l asz--Sobolev space $\dot{M}^{s,p}$  in our framework, defined in Section \ref{s.function_spaces},  is general enough to encompass a wide range of Sobolev-type spaces. In particular, both the homogeneous Newton--Sobolev and fractional Sobolev spaces admit a bounded embedding into the homogeneous Haj\l asz--Sobolev space considered here. For the former embedding to hold, one needs to assume a suitable Poincar\'e inequality as in \cite{KlKoNg}, forcing the space to be connected. We refer to Remark~\ref{r.HN} for further discussion. The corresponding embedding for homogeneous fractional Sobolev spaces is developed in  Section~\ref{s.fract_app}.

 Throughout the paper, we use median values instead of integral averages to facilitate the analysis of functions in the homogeneous Haj\l asz--Sobolev space $\dot{M}^{s,p}$. Such a function may fail to be integrable on balls for small values of $s$ and $p$, hence the corresponding integral averages are not well-defined. The use of medians in this context is well established; we refer to \cite{AKM,GoKoZh,HeKoTu,JaTo} for related developments. See Sections  \ref{median definition} and \ref{medians at infinity} for a further discussion.

A distinctive feature of our approach in studying limits at infinity is a novel use of variational relative capacity, as introduced in Definition~\ref{d.relative}. In one of our main results, Theorem~\ref{t.main}, we establish conditions, under which a quasicontinuous representative of a function in the homogeneous Haj\l asz--Sobolev space admits a limit of the form
\[
\lim_{\substack{d(x,O)\to \infty \\ x\in X\setminus E(O)}} u(x)\,.
\]
Here $O\in X$ is a fixed basepoint and $E(O)$ is an exceptional set, outside
of which the limit of the function $u$  exists.  
 To quantify the smallness of such an exceptional set  at infinity, we utilize the variational relative capacity.  These exceptional sets in our framework are called $(s,p)$-thin at infinity, and they are discussed in Section~\ref{s.thin}.
 The novel use of relative capacity, instead of Hausdorff content as in \cite{KlKoNg,AKM}, provides a finer exceptional set, resulting in sharper and more precise analytical conclusions. We refer the reader to Section~\ref{s.h} and Theorem~\ref{e.hthin} for detailed comparisons. As an application, we extend and complement the results in \cite{KlKoNg} and \cite{AKM}. For  precise formulations, we refer to Remark~\ref{r.HN} and Theorem~\ref{W,s,p,q 1}, respectively.

\section{Metric spaces}\label{e.preliminaries}

We introduce our standing assumptions on
the metric space and the measure, which are maintained throughout the paper.
For further details, we refer to \cite{BB,MR1800917,MR3363168}.
 Let  $(X,d,\mu)$ be a  metric measure space with $\diam(X)=\infty$. Here $\mu$ is a positive  complete Borel measure on $X$ such that \[0<\mu(B(x,r))<\infty\] for all $x\in X$ and $r>0$, where  $B(x,r)=\{y\in X\,:\, d(y,x)<r\}$  is the open ball of radius $r>0$ and center $x\in X$.
The measure $\mu$ is  assumed to be doubling, that is, there exists a constant $c_{\mu}\geq1$ such that, whenever $x\in X$ and $r>0$, we have
\begin{equation}\label{doubling}
    \mu(B(x,2r))\leq c_{\mu}\,\mu(B(x,r)).
\end{equation}
We also assume that 
the measure $\mu$ satisfies
the reverse doubling condition, that is,
there are constants  $0<c_R<1$   and $\kappa>1$   such that
\begin{equation}\label{reverse doubling}
\mu(B(x,r))\le c_R\,   \mu(B(x,\kappa r))  
\end{equation}
for all $x\in X$ and all $r>0$.

The above
assumptions are tacitly held throughout the paper. The space $X$ is separable under these assumptions,
see \cite[Proposition~1.6]{BB}.
It  follows from the reverse doubling condition~\eqref{reverse doubling} that  $\mu(X)=\infty$. 
If $X$ is connected, then \eqref{reverse doubling} for any $\kappa>1$ 
follows from the doubling condition \eqref{doubling} and the
unboundedness of $X$; 
see for instance~\cite[Lemma~3.7]{BB}. However,
we do not assume in this paper that $X$ is connected  unless otherwise specified. 
Since $\mu$ is  doubling  in our setting, the reverse doubling condition \eqref{reverse doubling} is equivalent
with the so-called uniform perfectness of $X$, we refer to 
\cite[Section 11]{MR1800917} for background on this condition and examples.

We  also need the following iterated versions
of the doubling and reverse doubling conditions.
Iteration of the doubling condition \eqref{doubling} gives an exponent $Q>0$ and a constant $c_{Q}>0$, both only depending on $c_{\mu}$, such that
\begin{equation}\label{iteration doubling}
   \frac{\mu(B(y,r))}{\mu(B(x,r_0))}\geq c_{Q}\left(\frac{r}{r_0}\right)^{Q}
\end{equation}
whenever $x\in X$, $y\in B(x,r_0)$ and $0<r<r_0<\infty$. Condition \eqref{iteration doubling} holds for $Q=\log_2 c_\mu$, we refer to \cite[Lemma 3.3]{BB} for details.
On the other hand, the doubling property and iteration of the reverse doubling condition \eqref{reverse doubling} gives an exponent $\sigma>0$ and a constant $c_{\sigma}>0$, both only depending on $c_R$, such that the  quantitative reverse doubling condition
\begin{equation}\label{iteration reverse}
   \frac{\mu(B(x,r))}{\mu(B(x,r_0))}\leq c_{\sigma}\left(\frac{r}{r_0}\right)^{\sigma}
\end{equation}
holds for all $x\in X$ and $0<r<r_0<\infty$.  We refer
to the proof of \cite[Corollary 3.8]{BB}.

We use the following familiar notation: \[
\vint_{E} u(y)\,d\mu(y)=\frac{1}{\mu(E)}\int_E u(y)\,d\mu(y)
\]
is the integral average of a non-negative measurable function $u$ over a measurable set $E\subset X$ 
with a finite positive measure, that is, $0<\mu(E)<\infty$. 
The closure of a set $E\subset X$ is denoted by $\iol{E}$. In particular, if $B=B(x,r)\subset X$ is a ball,
then the notation $\iol{B}=\iol{B(x,r)}$ refers to the closure of the ball $B$.  
The characteristic function of a set $E\subset X$ is denoted by $\mathbf{1}_{E}$; 
that is, $\mathbf{1}_{E}(x)=1$ if $x\in E$
and $\mathbf{1}_{E}(x)=0$ if $x\in X\setminus E$.

  Next we define certain annuli, whose scaling
is determined by the constant $\kappa>1$ in \eqref{reverse doubling}. These annuli
are often used in this paper, and therefore
we use the $\kappa$-scaling frequently.  
Assume that $O\in X$ and $r>0$. Then we write
  $A_r(O)=B(O,\kappa r)\setminus B(O,r)$.   When $\Lambda\ge 1$ is also given, we write
\[\Lambda A_r(O)=  B(O,\Lambda \kappa r)  \setminus B(O,r/\Lambda)\,.\]
It follows from \eqref{reverse doubling} that 
\begin{equation}\label{annuli ball comparison}
  (1-c_R) \mu(B(O, \Lambda \kappa r ))\leq\mu(\Lambda A_r(O))\leq \mu(B(O,\Lambda \kappa r)) 
\end{equation}
for all $O\in X$, $r>0$ and $\Lambda \ge 1$. The point $O\in X$ is called  the basepoint.

We will use the following notational convention:
$C({\ast,\dotsb,\ast})$ denotes a positive constant which quantitatively 
depends on the quantities indicated by the $\ast$'s but whose actual
value can change from one occurrence to another, even within a single line.

\section{Medians}\label{median definition}

   We work with measurable functions that need not be integrable on balls.
Therefore we use medians instead of integral averages.
In this section, we  recall useful lemmata for medians from \cite{AKM,GoKoZh,HeKoTu}. We also refer to \cite{Fuji,JaTo}.  

Let $L^0(X)$ denote the set of all
measurable functions $u\colon X\to [-\infty,\infty]$ that are finite almost everywhere on $X$.
Let  $u\in L^0(X)$ and let  $E\subset X$ be a set of finite positive measure. Define the median,  or the  $\frac{1}{2}$-median, of a function $u$ over $E$ as
\[
m_u(E)= \inf\left\{ a\in\R \,:\, \mu\left(\{x\in E \,:\, u(x)>a\}\right) < \frac{\mu(E)}{2}\right\}.
\]
By using continuity of measure, it is easy to show that $-\infty < m_u(E)<\infty$, that is, $m_u(E)$ is finite.
We also define the sub- and superlevel sets $E^{u,-}$ and $E^{u,+}$ as follows
\begin{equation}\label{Eplus E minus}
    E^{u,-}= \left\{ x\in E \,:\, u(x)\leq m_u(E) \right\}
\quad 
\text{and}
\quad
E^{u,+}= \left\{ x\in E \,:\, u(x)\geq m_u(E) \right\}.
\end{equation}
It is again easy to show that $2\mu(E^{u,-})\geq \mu(E)$ and $2\mu(E^{u,+})\geq \mu(E)$.

 We recall the following basic
 properties of medians from \cite[Lemma~2.7]{HeKoTu}.

\begin{lemma}\label{median prps}
Let $u\in L^0(X)$ and let  $E\subset X$ be a set of finite positive measure.
Then the median has the following properties:
  \begin{itemize}
      \item[(a)] $|m_u(E)|\leq m_{|u|}(E)$.
      \item[(b)] If $c\in\R$, then $m_u(E)+c= m_{u+c}(E)$.
      \item[(c)] If $0<p<\infty$, then
      \begin{equation*}
          m_{|u|}(E)\leq\left(2\intav_E|u(x)|^p\, d\mu(x)\right)^\frac{1}{p}.
      \end{equation*}
  \end{itemize}
\end{lemma}

Next we give two  lemmata for
differences involving medians.   In the case of $\R^n$, these results can already be found in 
\cite[Lemma 2.4, Lemma 2.5]{AKM}. We recall the  proofs for convenience.  

\begin{lemma}\label{med_another lemma}
Let $u\in L^0(X)$ and let $E\subset X$ be a set of finite positive measure. Then, for every $c\in\R$, there exists a set $\tilde{E}\subset E$ such that $2\mu(\tilde{E})\geq \mu(E)$ and 
		$$|m_u(E)-c|\leq|u(x)-c|$$whenever $x\in\tilde{E}$.
\end{lemma}

\begin{proof}
    Fix $c\in\R$. First suppose that $m_u(E)\geq c$. 
    Let $\tilde E=E^{u,+}\subset E$. Then $2\mu(\tilde E)\ge \mu(E)$ and,
    for every $x\in \tilde E$, we have
\begin{equation*}
    \lvert m_u(E)-c\rvert= m_u(E)-c \leq u(x) -c\leq |u(x)-c|\,.
\end{equation*}
Thus, the claim holds in the case $m_u(E)\geq c$.
 Similarly, in the case $m_u(E)< c$ we define $\tilde{E}=E^{u,-}$. Adapting the above argument yields the desired conclusion also in this case. 
\end{proof}

\begin{lemma}\label{l1}
Let $u\in L^0(X)$ and let $E,F\subset X$ be sets of finite positive measure. Then there exists measurable subsets $\tilde{E}\subset E$ and $\tilde{F}\subset F$  such that $2\mu(\tilde{E})\geq\mu(E)$, $2\mu(\tilde{F})\geq\mu(F)$ and 
$$|m_u(E)-m_u(F)|\leq|u(x)-u(y)|$$for all $x\in\tilde{E}$ and  $y\in\tilde{F}$.
\end{lemma}

\begin{proof}
First, assume that \[m_u(E)\leq m_u(F)\,.\] Let $\tilde{E}=E^{u,-}\subset E$ and $\tilde{F}=F^{u,+}\subset F$.
Then
$2\mu(\tilde{E})\geq \mu({E})$ and  $2\mu(\tilde{F})\geq \mu({F})$.
Also, for all $x\in\tilde{E}$ and $y\in\tilde{F}$, we get  from \eqref{Eplus E minus}  that
\begin{align*}
|m_u(E)-  m_u(F)|= m_u(F)-  m_u(E)
\leq u(y)-u(x)
= |u(x)-u(y)|.
\end{align*}
Thus, the claim holds in the case $m_u(E)\leq m_u(F)$.
 Similarly, in the case \[m_u(F)< m_u(E)\] we define $\tilde{E}=E^{u,+}$ and $\tilde{F}=F^{u,-}$. Clearly, adapting the above argument yields the desired conclusion also in this case. \end{proof}

\section{Haj{\l}asz--Sobolev functions and quasicontinuity}\label{s.function_spaces}

  We use Haj{\l}asz--Sobolev functions \cite{Ha,MR2039955} and the associated notion of Haj{\l}asz capacity \cite{MR1752853,HeKoTu}. This  capacity is used
to define quasicontinuous Haj{\l}asz--Sobolev functions, whose pointwise limits at infinity are of main interest to us in this work.  
	
	\begin{definition}\label{d.haj}
    Let $0<s\leq1$ and let $u\in L^0(X)$. A measurable function $g\colon X\to [0,\infty]$ is an $s$-gradient of $u$, if there exists a set $E\subset X$ with $\mu(E)=0$ such that for all $x,y\in X\setminus E$,
    \begin{equation}\label{s grad define}
       |u(x)-u(y)|\leq d(x,y)^s(g(x)+g(y))\,. 
    \end{equation}
    The collection of all $s$-gradients of $u$ is denoted by $\mathcal{D}_s(u)$.
    \end{definition}
 
 \begin{definition}
    Let $0<s\le 1$ and $0<p<\infty$. The homogeneous  Haj{\l}asz--Sobolev  space $\dot{M}^{s,p}(X)$ consists of functions $u\in L^0(X)$ such that    $\D_s(u)\cap L^p(X)\not=\emptyset$. This vector space is equipped with the quasi-seminorm  
\begin{equation*}
  \|u\|_{\dot{M}^{s,p}(X)}=\inf_{g\in \D_s(u)}\|g\|_{L^p(X)} \,.  
\end{equation*} 
The  Haj{\l}asz--Sobolev  space ${M}^{s,p}(X)$ is  $\dot{M}^{s,p}(X)\cap L^p(X)$ equipped with the quasi-seminorm
\begin{equation*}
\|u\|_{{M}^{s,p}(X)}=\|u\|_{L^p(X)}+\|u\|_{\dot{M}^{s,p}(X)}.
\end{equation*}
\end{definition}

  Throughout this paper, we will consider $L^0(X)$, $\dot{M}^{s,p}(X)$ and ${M}^{s,p}(X)$ to be vector spaces of functions defined everywhere, instead of
equivalence classes defined $\mu$-almost everywhere. 
We nevertheless say that   $v\in L^0(X)$ is a representative of a function $u\in L^0(X)$, if $v(x)=u(x)$ holds for $\mu$-almost every $x\in X$.
We remark that $\|v\|_{\dot{M}^{s,p}(X)}=\|u\|_{\dot{M}^{s,p}(X)}$ and $\|v\|_{{M}^{s,p}(X)}=\|u\|_{{M}^{s,p}(X)}$ if $v$ is a representative of $u$.

We are mainly interested in fine properties of Haj{\l}asz--Sobolev functions, and therefore
we also need the notion of capacity,   see \cite{MR1752853,HeKoTu}.  

 \begin{definition}\label{d.h_cap}
Let $0<s\leq1$ and $0<p<\infty$. 
The $M^{s,p}$-capacity of a set $E\subset X$ is 
\begin{equation*}
   \mathrm{Cap}_{M^{s,p}}(E)=\inf\left\{\|u\|_{M^{s,p}(X)}^p\,:\, u\in M^{s,p}(X)\text{ and } u\geq1 \text{ on a neighbourhood of $E$}\right\}.
\end{equation*}
If there are no test functions that satisfy the above conditions, we set $\mathrm{Cap}_{M^{s,p}}(E)=\infty$. We say that a property holds $M^{s,p}$-quasieverywhere, if it holds outside a set of $M^{s,p}$-capacity zero.
\end{definition}

Now we list some of the basic properties of the $M^{s,p}$-capacity from \cite{HeKoTu}.

\begin{lemma}\label{l.cap_basic}
Let $0<s\le 1$ and $0<p<\infty$. The $M^{s,p}$-capacity
has the following properties:
\begin{itemize}
    \item[(a)] If $E\subset F\subset X$, then $\mathrm{Cap}_{M^{s,p}}(E)\leq \mathrm{Cap}_{M^{s,p}}(F)$.
    \item[(b)] There exists a constant $c\geq1$ such that 
    \[
        \mathrm{Cap}_{M^{s,p}}\left(\bigcup_{i\in\N}E_i\right)\leq c\sum_{i\in\N}\mathrm{Cap}_{M^{s,p}}(E_i)
\]
    whenever $E_i\subset X$ for all $i\in\N$.
    \item[(c)] The $M^{s,p}$-capacity is an outer capacity, that is,     \begin{equation*}
       \mathrm{Cap}_{M^{s,p}}(E)=\inf\{\mathrm{Cap}_{M^{s,p}}(U)\,:\, E\subset U\text{ and } U \text{ is open}\}. 
    \end{equation*}
    \end{itemize}
\end{lemma}

\begin{proof}
(a) This follows immediately from the definition of $M^{s,p}$-capacity.
(b) This follows from the case $q=\infty$ of \cite[Lemma~6.4]{HeKoTu} since $M^{s,p}(X)=M^{s}_{p,\infty}(X)$ using their notation, we also refer to  \cite[Proposition 2.1]{MR2764899}.
  (c) This is \cite[Remark~6.3]{HeKoTu}.  
    \end{proof}

We will mostly work with $M^{s,p}$-quasicontinuous functions. These are defined next.

\begin{definition}\label{d.qc}
    A function $u\in L^0(X)$ is  $M^{s,p}$-quasicontinuous in $X$, if for all $\varepsilon>0$ there exists a set $E\subset X$ such that $\mathrm{Cap}_{M^{s,p}}(E)<\varepsilon$ and the restriction of $u$ to $X\setminus E$ is finite and  continuous. 
\end{definition}

We observe that the set $E\subset X$ in Definition \ref{d.qc} can be assumed to be open in $X$. This is an immediate consequence of Lemma \ref{l.cap_basic}(c).

Next we  show that if $v$ is
a quasicontinuous representative of $u$, then $u=v$ not only almost everywhere, but quasieverywhere in $X$. 
This type of results
are well known for Sobolev spaces, in general. 
We refer to 
\cite{N} for a corresponding result in Haj{\l}asz--Besov spaces.
  We
provide a proof in our setting 
for convenience of the reader. 
To this end, we need the following two lemmata
that are adapted from \cite{Ki}. See also \cite{KiLeVa}.

\begin{lemma}\label{prop0}
Let $0<s\leq1$ and $0<p<\infty$. For every open set $G\subset X$ and every $E\subset X$ with $\mu(E)=0$, we have $\mathrm{Cap}_{M^{s,p}}(G)=\mathrm{Cap}_{M^{s,p}}(G\setminus E)$.
\end{lemma}

\begin{proof}
By monotonicity, or Lemma \ref{l.cap_basic}(a), we have $\mathrm{Cap}_{M^{s,p}}(G\setminus E)\leq\mathrm{Cap}_{M^{s,p}}(G) $. To prove the converse inequality, we may clearly assume  $\mathrm{Cap}_{M^{s,p}}(G\setminus E)<\infty$. Let $\varepsilon>0$. Then there exists a test function $u$ for $\mathrm{Cap}_{M^{s,p}}(G\setminus E)$ such that
\[
\|u\|_{M^{s,p}(X)}^p\leq \mathrm{Cap}_{M^{s,p}}(G\setminus E)+\varepsilon
\]
and $u\ge 1$ on some open set $U\subset X$ such that $G\setminus E\subset U$. We define $\tilde{u}$ so that $\tilde{u}=u$ on $X\setminus E$ and $\tilde{u}=1$ on $E$. Then $\tilde{u}\in M^{s,p}(X)$, since $\mu(E)=0$ and therefore
$\tilde u$ is a representative of $u$. Observe also that $\tilde{u}\geq 1$ on $G\cup U$. Hence, we obtain
\[\mathrm{Cap}_{M^{s,p}}(G)\leq\mathrm{Cap}_{M^{s,p}}(G\cup U)\leq  \|\tilde{u}\|_{M^{s,p}(X)}^p\leq  \mathrm{Cap}_{M^{s,p}}(G\setminus E)+\varepsilon\,.\]
Thus the claim follows when $\varepsilon\to 0$.
\end{proof}

\begin{lemma}\label{prop0-1}
Let $0<s\leq1$ and $0<p<\infty$. Let  $u,v\in L^0(X)$ be  $M^{s,p}$-quasicontinuous functions such that  $u=v$ $\mu$-almost everywhere. 
Then $u=v$ holds $M^{s,p}$-quasieverywhere.
\end{lemma}

\begin{proof}
Let $c\geq1$ be the subaddivitivity constant in Lemma \ref{l.cap_basic}(b) and let $\varepsilon>0$ be given.
Let $ E_1,E_2\subset X$ be open sets such that $\mathrm{Cap}_{M^{s,p}}(E_1),\mathrm{Cap}_{M^{s,p}}(E_2)<\frac{\varepsilon}{2c}$ and the restrictions of $u$ to $X\setminus E_1$ and $v$ to $X\setminus E_2$ are continuous, respectively. 
Define an open set $E=E_1\cup E_2$. By Lemma \ref{l.cap_basic}(b), we have  $\mathrm{Cap}_{M^{s,p}}(E) <\varepsilon$. Observe that the difference $u-v$ 
 is continuous in $X\setminus E$, and hence the set
\[
\{x\in X\setminus E\,:\, u(x)\neq v(x)\}= \{x\in X\setminus E\,:\, (u-v)(x)\neq 0\}
\]
is open in the relative topology of $X\setminus E$. Thus, there exists an open set $U\subset X$ such that
\[
\{x\in X\setminus E\,:\, u(x)\neq v(x)\} = U\cap (X\setminus E) =U\setminus E\,. 
\]
On the other hand, we have
\[\{x\in X\,:\, u(x)\neq v(x)\}\subset E\cup \{x\in X\setminus E\,:\, u(x)\neq v(x)\} = E\cup (U\setminus E)=U\cup E\,.
\]
and
\[
\mu(U\setminus E)=\mu(\{x\in X\setminus E\,:\, u(x)\neq v(x)\})=0\,,
\]
since $u=v$ $\mu$-almost everywhere in $X$.
By monotonicity and Lemma~\ref{prop0},
\begin{align*}
\mathrm{Cap}_{M^{s,p}}(\{x\in X\,:\, u(x)\neq v(x)\})&\le \mathrm{Cap}_{M^{s,p}}(U\cup E)\\&=\mathrm{Cap}_{M^{s,p}}((U\cup E)\setminus (U\setminus E))=\mathrm{Cap}_{M^{s,p}}(E)<\varepsilon. 
\end{align*}
Hence the claim follows when $\varepsilon\to0$.
\end{proof}

The following theorem is a slight reformulation of \cite[Theorem~1.2]{HeKoTu}. It shows that a
function $u\in\dot{M}^{s,p}(X)$ has an explicit $M^{s,p}$-quasicontinuous representative $u^*$ that can be defined
by using certain limits of medians.

\begin{theorem}\label{quasicont repesent}
    Let $0<s\leq1$ and $0<p<\infty$. Let $u\in\dot{M}^{s,p}(X)$ and define \[u^*(x)=\limsup_{r\to 0}m_u(B(x,r))\] for all $x\in X$. Then there exists a set $E\subset X$ with $\mathrm{Cap}_{M^{s,p}}(E)=0$ such that the limit  
    \begin{equation}\label{represent}
         \lim_{r\to 0}m_u(B(x,r))=u^*(x)
    \end{equation}
    exists for every $x\in X\setminus E$. Moreover, $u^*$ is an $M^{s,p}$-quasicontinuous representative of $u$.
    \end{theorem}
    
\section{Limits of medians at infinity}\label{medians at infinity}

  We are interested in  limits of quasicontinuous Haj{\l}asz--Sobolev functions
at infinity. The first step is to study
the limits of medians   $m_u(A_{\kappa^j}(O))$,   where $O\in X$
is a fixed basepoint and
\[
  A_{\kappa^j}(O)=  B(O,\kappa^{j+1})\setminus B(O,\kappa^{j}) \,,\qquad j\in\N\,.
\]
  Inspired by \cite[Section 4]{AKM}, we prove the existence and basepoint-independence of the limit of these medians.

First we consider the easier case of $u\in M^{s,p}(X)\subset L^p(X)$.

\begin{lemma}\label{median lp cong}
Let $0<p<\infty$ and $u\in L^p(X)$. Then  
\begin{equation}\label{median convg lP}
    \lim_{j\to\infty}m_u(  A_{\kappa^j}(O))=0
\end{equation}
 for every $O\in X$.
\end{lemma}
\begin{proof}

 Fix $O\in X$ and let $j\in \N$. By using Lemma~\ref{median prps} and \eqref{annuli ball comparison}, we obtain
 \begin{equation*}\begin{split}
    |m_u( A_{\kappa^j}(O))|&\leq m_{\lvert u\rvert}( A_{\kappa^j}(O))\le \left(\frac{2}{\mu( A_{\kappa^j}(O))}\int_{ A_{\kappa^j}(O)} |u(x)|^p \,d\mu(x)\right)^\frac{1}{p} \\
           &\leq \frac{C(c_R,p)}{\mu(B(O,\kappa^j))^{1/p}}\left(\int_{X} |u(x)|^p \,d\mu(x)\right)^\frac{1}{p}\,.
     \end{split}\end{equation*}
     Since $\mu(X)=\infty$ and $u\in L^p(X)$, the right-hand side tends to zero, as $j\to\infty$.
\end{proof}

Next we turn to the case $u\in \dot{M}^{s,p}(X)$. This requires more work, and we begin with the following lemma that is
used several times throughout the paper.

\begin{lemma}\label{med_int}
Let $0<s\leq1$ and $0<p<\infty$, and let
$u\in L^0(X)$ and $g\in\mathcal{D}_s(u)$. If $E,F\subset X$ are sets of finite positive measure, then
		\begin{align}\label{2.1.1}
			|m_u(E)-m_u(F)|^p\leq C(p)D(E,F)^{sp} \left(\intav_{E}g^p(x)\, d\mu(x)+\intav_{
    F }g^p(y)\, d\mu(y)\right)\,,
		\end{align}
		where \[   D(E,F)=\sup\{d(x,y)\,:\, x\in E\text{ and }y\in F\}\,. \]    
	\end{lemma}
	
	\begin{proof}
By using Lemma~\ref{l1}, we obtain two  measurable subsets $\tilde{E}\subset E$ and $\tilde{F}\subset F$ such that $2\mu(\tilde{E})\geq \mu(E)$, $2\mu(\tilde{F})\geq \mu(F)$ and 
\[|m_u(E)-m_u(F)|^p\leq|u(x)-u(y)|^p\]
for all $x\in\tilde{E}$ and $y\in\tilde{F}$. By integration, we obtain
		\begin{align*}
			|m_u(E)-m_u(F)|^p
			&\leq\intav_{\tilde{E}}\intav_{\tilde{F}}|u(x)-u(y)|^p\,d\mu(y)\,d\mu(x)\\
			&=\frac{1}{\mu(\tilde{E})\mu(\tilde{F})}\int_{\tilde{E}}\int_{\tilde{F}}|u(x)-u(y)|^p\,d\mu(y)\,d\mu(x)\\
			&\leq\frac{4}{\mu(E)\mu(F)}\int_{E}\int_{F}|u(x)-u(y)|^p\,d\mu(y)\,d\mu(x)\\
			&\leq4\intav_{E}\intav_{F}|u(x)-u(y)|^p\,d\mu(y)\,d\mu(x).
		\end{align*}
       Since $g\in\mathcal{D}_s(u)$, we get
        \begin{align*}
            |m_u(E)-m_u(F)|^p&\leq 4D(E,F)^{sp}\intav_{E}\intav_{
   F}(g(x)+g(y))^p\,\, d\mu(y)\, d\mu(x)\\
   &\leq C(p)D(E,F)^{sp} \intav_{E}\intav_{
    F}(g^p(x)+g^p(y))\, d\mu(y)\, d\mu(x)\\
   &\leq C(p)D(E,F)^{sp}\intav_{E}\left(g^p(x)+\intav_{
    F }g^p(y)\, d\mu(y)\right)\, d\mu(x)\\
    &\leq C(p)D(E,F)^{sp} \left(\intav_{E}g^p(x)\, d\mu(x)+\intav_{
    F }g^p(y)\, d\mu(y)\right).
        \end{align*}
        Hence, the claim follows.
	\end{proof}

Next we establish a counterpart
of Lemma \ref{median lp cong} for functions $u\in \dot{M}^{s,p}(X)$. Here we need the additional assumption
$sp<\sigma$, where
$\sigma>0$ is an exponent in the quantitative
reverse doubling condition \eqref{iteration reverse}.
Recall that such a positive exponent always exists due to our standing assumptions described in Section \ref{e.preliminaries}.

\begin{lemma}\label{l4}
Let $0<s\leq1$ and $0<p<\infty$ be such that $sp<\sigma$, where $\sigma>0$ is as in \eqref{iteration reverse}.
Let $u\in \dot{M}^{s,p}(X)$. Then there exists a constant $c\in\R$ such that 
\begin{equation}\label{median convergence_2}
    \lim_{j\to\infty}m_u(  A_{\kappa^j}(O))=c
\end{equation}
for all $O\in X$. That is, the limit always exists and it is independent
of $O\in X$.
\end{lemma}
\begin{proof}
Since $u\in \dot{M}^{s,p}(X)$, there exists a function $g\in \mathcal{D}_s(u)\cap L^p(X)$. 
  First we  show that the limit in \eqref{median convergence_2} exists for every $O\in X$. Therefore, we fix $O\in X$ and $k\in\N$. 
 Lemma~\ref{med_int} followed by inequalities \eqref{annuli ball comparison} and \eqref{iteration reverse} gives
    \begin{align*}
        & |m_u( A_{\kappa^{k+1}}(O))-m_u( A_{\kappa^k}(O))|^p   \\
   & \leq C(s,p,\kappa)\kappa^{ksp} \left(\frac{1}{\mu({A}_{\kappa^{k+1}}(O))}\int_{{A}_{\kappa^{k+1}}(O)}g^p(x)\, d\mu(x)+\frac{1}{\mu({A}_{\kappa^{k}}(O))}\int_{
   {A}_{\kappa^{k}}(O) }g^p(y)\, d\mu(y)\right)\nonumber\\
    &\leq\frac{C(s,p,\kappa,c_R)}{\mu(B(O,\kappa^k))}\kappa^{ksp}\left(\int_{X}g^p(x)\, d\mu(x)+\int_{X }g^p(y)\, d\mu(y)\right) \nonumber\\
    &\leq \frac{C(s,p,\kappa,c_R,c_{\sigma})}{\mu(B(O,1))}\kappa^{k(sp-\sigma)} \int_{X}g^p(x)\, d\mu(x)=C^p\,\kappa^{k(sp-\sigma)} \int_{X}g^p(x)\, d\mu(x)\,,
 \end{align*}  
 where the constant $C>0$ is defined by the last equality.
We can conclude from above that, for all $k\in\N$,
\begin{equation}\label{med diff}
     |m_u( A_{\kappa^{k+1}}(O))-m_u( A_{\kappa^k}(O))|\leq  C\kappa^{k(s-\frac{\sigma}{p})} \|g\|_{L^p(X)}<\infty\,.
\end{equation} 
Hence, for all $i,j\in\N$ with $i>j$, we have  by \eqref{med diff},
\begin{align*}
       & |m_u(A_{\kappa^i}(O))-m_u( A_{\kappa^j}(O))|\\&\qquad
    \leq\sum_{k=j}^{i-1}|m_u(A_{\kappa^{k+1}}(O))-m_u(A_{\kappa^k}(O))|
     \leq C \|g\|_{L^p(X)}\sum_{k=j}^{\infty}\kappa^{k(s-\frac{\sigma}{p})}\,.
    \end{align*}
Since $sp<\sigma$, 
the sequence $(m_u(  A_{\kappa^j}(O)))_{j\in\N}$ is 
Cauchy, and
therefore it converges.

We have shown that the limit in the left-hand side of \eqref{median convergence_2} 
exists for all $O\in X$.
It remains to show that this limit is independent of $O\in X$.
For this purpose, we fix 
		two points $O_1,O_2\in X$.
        Suppose that $j\in\N$ satisfies $j\geq \log_\kappa d(O_1,O_2)$. Then 
        \[
        D( A_{\kappa^j}(O_1), A_{\kappa^j}(O_2))=  \sup\{d(x,y)\,:\, x\in  A_{\kappa^j}(O_1)\text{ and }y\in  A_{\kappa^j}(O_2)\} \leq 3\kappa^{j+1}
        \]
         and Lemma~\ref{med_int}, followed by inequalities \eqref{annuli ball comparison} and \eqref{iteration reverse}, gives 
		\begin{align*}
		|m_u&( A_{\kappa^j}(O_1))-m_u( A_{\kappa^j}(O_2))|^p\\
        &\leq C(s,p,\kappa)\kappa^{jsp}\left(\intav_{ A_{\kappa^j}(O_1)}g^p(x)\, d\mu(x)+\intav_{
    A_{\kappa^j}(O_2) }g^p(y)\, d\mu(y)\right)\\
		&\leq \left(\frac{C(s,p,\kappa,c_R,c_{\sigma})}{\mu(B(O_1,1))}+\frac{C(s,p,\kappa,c_R,c_{\sigma})}{\mu(B(O_2,1))}\right)\lVert g\rVert_{L^p(X)}^p \kappa^{j(sp-\sigma)}<\infty\,.
		\end{align*}
		 By taking  $j\to\infty$ and using the fact
		 that $sp<\sigma$, we 
		 find that 
		 \[
		 \lim_{j\to\infty}m_u(  A_{\kappa^j}(O_1))=\lim_{j\to\infty}m_u(  A_{\kappa^j}(O_2))\,,
\]
and therefore we can choose $c=\lim_{j\to\infty}m_u(  A_{\kappa^j}(O_1))$ in  \eqref{median convergence_2}, that is, the value of the limit therein is independent of $O\in X$.
		 \end{proof}
		 
		  \section{A relative capacity}

The $M^{s,p}$-capacity suits well for quantifying many fine properties of Haj{\l}asz--Sobolev functions. On the other hand, it is lacking in   scaling   and locality properties. To address these defects,  in Definition \ref{d.relative}, we deploy variant of a relative capacity
that is more suitable for describing  certain exceptional sets, outside of which we may approach  infinity. These 
so-called $(s,p)$-thin sets at infinity are defined in Section \ref{s.thin}.
We remark in passing that there are several  seemingly different variants of relative capacities of fractional smoothness in metric measure spaces,   see for example \cite{MR4894884,MR4870861} for  recent comparisons among some of them.    

\begin{definition}\label{d.relative}
 Let $0<s\leq1$ and $0<p<\infty$. Let $F\subset X$ be a  measurable   set of  diameter $0<\diam(F)<\infty$   and let $E\subset F$. We define a   relative   capacity 
\[
\text{cap}_{s,p}(E,F)=\inf_{v\in\mathcal{A}_{s,p}(E)}\left\{ { \frac{||v||_{L^p(F)}^p}{ \diam(F)^{sp}} + \inf_{g\in \D_s(v)} ||g||_{L^p(F)}^p}\right\}\,,
\]
  where $\mathcal{A}_{s,p}(E)\subset \dot{M}^{s,p}(X) $ is the set of all $M^{s,p}$-quasicontinuous functions $v\in\dot{M}^{s,p}(X)$ such that $v\ge 1$ in $E$. 
 \end{definition}

\begin{remark}
Suppose that $s$, $p$, $E$, $F$ are as in Definition \ref{d.relative}.
  Observe that $\mathbf{1}_X\in \mathcal{A}_{s,p}(E)$
and $\D_s(v)\cap L^p(F)\not=\emptyset$ for all $v\in \mathcal{A}_{s,p}(E)$.
Hence, the relative capacity in Definition \ref{d.relative}  is well defined and finite.  
Moreover, we clearly have 
\[
\frac{\mu(E)}{\diam(F)^{sp}}\le \text{cap}_{s,p}(E,F)\le \frac{\mu(F)}{\diam(F)^{sp}}\,.
\]
More refined lower bounds are  established in Section \ref{s.h}, see Lemma \ref{homogeneous Cap and Hauss est}.
\end{remark}

For the following   capacitary weak type estimate,   we recall
that \[
  \Lambda  A_{\kappa^j}(O)=B(O,\Lambda \kappa^{j+1})\setminus B(O,\kappa^j/\Lambda) 
\]
 whenever $\Lambda \ge 1$, $O\in X$ and $j\in\N$.
 As a special case, we have $ A_{\kappa^j}(O)=1  A_{\kappa^j}(O)$.

\begin{lemma}\label{l3}
Let $0<s\leq1$ and $0<p<\infty$, and let  $u\in \dot{M}^{s,p}(X)$ be an $M^{s,p}$-quasicontinuous function. Assume that $O\in X$, $j\in\N$ and  $t>0$. Define a superlevel set
\[E_t=   \{x\in   A_{\kappa^j}(O) \,:\, |u(x)-m_u(  A_{\kappa^j}(O))|>t\}\,.
\]
    Then for every $\Lambda\ge 1$ and $g\in\mathcal{D}_s(u)$, we have
    $$\mathrm{cap}_{s,p}(E_t,\Lambda  A_{\kappa^j}(O))\leq \frac{C}{t^p}\int_{\Lambda  A_{\kappa^j}(O)}g^p(x)\,d\mu(x)$$ 
    where $C=C(p,c_\mu,c_R,\Lambda)$.
\end{lemma}

\begin{proof}
Fix $\Lambda\ge 1$ and $g\in\mathcal{D}_s(u)$.
We define
\begin{align}\label{test function_1}
    v(x)= \frac{|u(x)-m_u(  A_{\kappa^j}(O))|}{t}\,,\qquad x\in X.
\end{align}
Then $v$ is $M^{s,p}$-quasicontinuous, since $u$ is such.
By Definition \ref{d.haj} 
there exists a set $A\subset X$ such that $\mu(A)=0$ and, for all $x,y\in X\setminus A$, 
\begin{equation}
\begin{split}\label{cap_est00}
    |v(x)-v(y)|&=\left|\frac{|u(x)-m_u(   A_{\kappa^j}(O))|}{t}-\frac{|u(y)-m_u(   A_{\kappa^j}(O))|}{t}\right|\\
    &\leq\frac{|u(x)-u(y)|}{t} \\
    &\leq \frac{d(x,y)^s}{t}(g(x)+g(y)).
    \end{split}
    \end{equation}
Therefore $\frac{g}{t}\in \mathcal{D}_s(v)$ for all $g\in \mathcal{D}_s(u)$ and thus $v\in\dot{M}^{s,p}(X)$ since $u\in \dot{M}^{s,p}(X)$. Observe that $v\ge 1$ in $E_t$. Hence, we can use $v$ as a test function for  $\mathrm{cap}_{s,p}(E_t,\Lambda  A_{\kappa^j}(O))$. This gives
\begin{align*}
    \mathrm{cap}_{s,p}(E_t,\Lambda  A_{\kappa^j}(O))&
    \leq\frac{||v||_{L^p(\Lambda  A_{\kappa^j}(O))}^p}{ \diam(\Lambda  A_{\kappa^j}(O))^{sp}}+\frac{||g||_{L^p(\Lambda  A_{\kappa^j}(O))}^p}{t^p}\,.
\end{align*}
Hence, it suffices to show that
       \begin{align}\label{test function_3}
\frac{||v||_{L^p(\Lambda  A_{\kappa^j}(O))}^p}{ \diam(\Lambda  A_{\kappa^j}(O))^{sp}}\leq \frac{C(p,c_\mu,c_R,\Lambda)}{t^p}\int_{
    \Lambda  A_{\kappa^j}(O)}g^p(x)\, d\mu(x).
\end{align}

In order to prove inequality \eqref{test function_3}, we denote
\begin{equation*}
  A^{-}= \left\{ x\in \Lambda  A_{\kappa^j}(O) \,:\, u(x)\leq m_u(  A_{\kappa^j}(O)) \right\}\end{equation*}
and
\begin{equation*}
    A^{+}= \left\{ x\in \Lambda  A_{\kappa^j}(O) \,:\, u(x)\geq m_u(  A_{\kappa^j}(O)) \right\}\,.
\end{equation*}
Note that $ A_{\kappa^j}(O)^{u,-}\subset {A^-}$ and $ A_{\kappa^j}(O)^{u,+}\subset A^+$ where $ A_{\kappa^j}(O)^{u,-}, A_{\kappa^j}(O)^{u,+}$ are defined as in \eqref{Eplus E minus}. Thus $\min\{\mu({A^-}),\mu({A^+})\}\geq \frac{\mu( A_{\kappa^j}(O))}{2}>0$ and 
therefore we can estimate
 \begin{align*}
 \int_{ {A^+}}|u(x)-m_u(  A_{\kappa^j}(O))|^p\, d\mu(x)
    &=\intav_{ {A^-}}
    \int_{{A^+}}|u(x)-m_u(  A_{\kappa^j}(O))|^p\, d\mu(x)\, d\mu(y)\\
    &\leq \intav_{ {A^-}}\int_{
   {A^+}}|u(x)-u(y)|^p\, d\mu(x)\, d\mu(y)\\
    &= \int_{
     {A^+}}\intav_{ {A^-}}|u(x)-u(y)|^p\, d\mu(y)\, d\mu(x)\,.
   \end{align*} 
Using the fact that  $g\in\mathcal{D}_s(u)$, followed by inequalities inequalities~\eqref{annuli ball comparison} and \eqref{iteration doubling}, we obtain
    \begin{align*}
     \frac{1}{\diam(\Lambda  A_{\kappa^j}(O))^{sp}}&\int_{{A^+}}|u(x)-m_u(  A_{\kappa^j}(O))|^p\, d\mu(x)\nonumber\\
     &
    \leq \int_{
     {A^+}}\intav_{ {A^-}}(g(x)+g(y))^p\, d\mu(y)\, d\mu(x)\nonumber\\
    &\leq C(p)\int_{
    {A^+}}\intav_{ {A^-}}(g^p(x)+g^p(y))\, d\mu(y)\, d\mu(x)\nonumber\\
    &\leq C(p)\left(\int_{
     {A^+}}g^p(x)\, d\mu(x)+\int_{
   {A^+}}\intav_{ {A^-}}g^p(y)\, d\mu(y)\, d\mu(x)\right)\nonumber\\
     &\leq C(p)\left(\int_{
   {A^+}}g^p(x)\, d\mu(x)+
   \frac{\mu( {A^+})}{\mu( {A^-})}\int_{{A^-}}g^p(y)\, d\mu(y)\, \right)\nonumber\\
    &\leq C(p)\left(\int_{
   \Lambda   A_{\kappa^j}(O)}g^p(x)\, d\mu(x)+
   \frac{2\mu( \Lambda  A_{\kappa^j}(O))}{\mu(  A_{\kappa^j}(O))}\int_{\Lambda  A_{\kappa^j}(O)}g^p(y)\, d\mu(y)\right)\nonumber\\
    &\leq C(p,c_\mu,c_R,\Lambda)\int_{
   \Lambda  A_{\kappa^j}(O)}g^p(x)\, d\mu(x).
       \end{align*}
       Observe that $\Lambda  A_{\kappa^j}(O)=A^+\cup A^-$. Hence, by interchanging ${A^+}$ and ${A^-}$ and arguing as above,  we obtain
  \begin{equation}\label{test function_2}
  \begin{split}
       \frac{1}{\diam(\Lambda  A_{\kappa^j}(O))^{sp}}&\int_{\Lambda  A_{\kappa^j}(O)}|u(x)-m_u(  A_{\kappa^j}(O))|^p\, d\mu(x)\\&\qquad\leq C(p,c_\mu,c_R,\Lambda)\int_{
   \Lambda  A_{\kappa^j}(O)}g^p(x)\, d\mu(x)\,.
   \end{split}
  \end{equation}
We can conclude from \eqref{test function_1} and \eqref{test function_2} that inequality \eqref{test function_3} holds.
\end{proof}

\section{Pointwise limits  at infinity}\label{s.thin}

In this section, we study the pointwise behavior of Haj{\l}asz--Sobolev functions at infinity. In particular, we show that a limit at infinity for such a  function exists, and it is independent of the basepoint $O\in X$. 
In fact,  the limit  is taken along the complement of an exceptional set $E(O)\subset X$, and this mode of convergence is made precise in the following definition.

\begin{definition}\label{d.conv}
Assume that $u\in L^0(X)$. Fix
a basepoint $O\in X$ and a set $E(O)\subset X$,
possibly depending on $O$. Then by writing
\[
\lim_{\substack{d(x,O)\to \infty \\ x\in X\setminus E(O)}} u(x)=c
\] we mean that for every $\varepsilon>0$ there exists $N\in\N$ such that $|u(x)-c|<\varepsilon$ whenever $x\in X\setminus E(O)$ and $d(x,O)>N$.
\end{definition}

%

We often use the following $(s,p)$-thin sets at infinity to define exceptional sets, outside of which we may approach infinity.

\begin{definition}\label{sp thiness}
    Let $0<s\leq1$,  $0<p<\infty$, and $O\in X$.
    A set $E(O)\subset X$ is $(s,p)$-thin at infinity, if 
    \begin{equation*}
       \lim_{m\to\infty}\sum_{j\geq m} \mathrm{cap}_{s,p}(E(O)\cap A_{\kappa^{j}}(O),\Lambda A_{\kappa^{j}}(O))=0
    \end{equation*}
    for all $\Lambda>1$.
\end{definition}

Two exponents
$Q>0$ and $\sigma>0$ satisfying \eqref{iteration doubling} and \eqref{iteration reverse}, respectively,
play a decisive role in the basic structure of $(s,p)$-thin sets at infinity, see
Remark \ref{r.neg} and  Remark \ref{r.pos}  for details.
Observe that we always have $\sigma \le Q$.
To summarize, if $sp>Q$ then the condition of $(s,p)$-thin
set at infinity may become vacuous, and if $sp\le \sigma$ then the complement of 
an $(s,p)$-thin set at infinity is unbounded. We will mainly work
in the second regime,  which is natural also for other reasons. We will
not explicitly exclude possibly vacuous cases from our study.

\begin{remark}\label{r.neg}
We assume that $X$ is connected, in addition to the other standing assumptions. Let $0<s\leq 1$ and $0<p<\infty$ be such that
$sp> Q$, where $Q$ is as in 
\eqref{iteration doubling}.
We now show that, under these assumptions, all subsets of 
$X$ are $(s,p)$-thin at infinity. Therefore this notion
becomes vacuous in this regime of parameters.
Fix $O\in X$, $E(O)\subset X$ and $\Lambda>1$.
Since $X$ is connected, we have
  $\diam(\Lambda A_{\kappa^{j}}(O))\ge (\kappa-1)\kappa^j$   for all $j\in\N$.
Now, if $m\in\N$, we use $\mathbf{1}_X$ as a test function for
each of the capacities below,  followed by inequality \eqref{iteration doubling}, to get
\begin{align*}
&\sum_{j\geq m} \mathrm{cap}_{s,p}(E(O)\cap A_{\kappa^{j}}(O),\Lambda A_{\kappa^{j}}(O))
\\&\quad \le \sum_{j\geq m} \frac{||\mathbf{1}_{X}||_{L^p(\Lambda A_{\kappa^{j}}(O))}^p}{\diam(\Lambda A_{\kappa^{j}}(O))^{sp}}
\le \frac{1}{(\kappa-1)^{sp}}\sum_{j\geq m} \frac{\mu( \Lambda A_{\kappa^{j}}(O))}{\kappa^{jsp}}\\&\quad \le \frac{1}{(\kappa-1)^{sp}}\sum_{j\geq m} \frac{\mu(B(O,\Lambda \kappa^{j+1}))}{\kappa^{jsp}}
\le  \frac{\mu(B(O,\kappa \Lambda))}{c_Q(\kappa-1)^{sp}}\sum_{j\geq m} 
\kappa^{j(Q-sp)}\,.
\end{align*}
Since $Q-sp<0$, the right-hand side tends to zero, as $m\to\infty$.
This shows that every set $E(O)\subset X$ is $(s,p)$-thin at infinity, as claimed.
\end{remark}

\begin{remark}\label{r.pos}
Let $0<s\leq 1$ and $0<p<\infty$ be such that
$sp\le \sigma$, where $\sigma$ is as in 
\eqref{iteration reverse}.
Suppose that $O\in X$ and $E(O)\subset X$ is $(s,p)$-thin at infinity.
Then we claim that $X\setminus E(O)$ is unbounded. In other words, the infinity is a limit point
of $X\setminus E(O)$ 
in the sense that for every $N\in\N$ there exists $x\in X\setminus E(O)$ such that $d(x,O)>N$. As a consequence, the limit in Definition \ref{d.conv} is unique, if it exists.

Now, we make an antithesis and assume that $X\setminus E(O)$ is bounded. Then there exists $N\in\N$ such
that $X\setminus E(O)\subset B(O,N)$. Thus, if $\Lambda> 1$ and $m>\log_\kappa(N)$, then by using \eqref{annuli ball comparison} and \eqref{iteration reverse} we obtain
\begin{align*}
&\sum_{j\geq m} \mathrm{cap}_{s,p}(E(O)\cap A_{\kappa^{j}}(O),\Lambda A_{\kappa^{j}}(O))=\sum_{j\geq m} \mathrm{cap}_{s,p}(A_{\kappa^{j}}(O),\Lambda A_{\kappa^{j}}(O))\\
&\quad \ge \sum_{j\geq m} \frac{||\mathbf{1}_{A_{\kappa^{j}}(O)}||_{L^p(\Lambda A_{\kappa^{j}}(O))}^p}{\diam(\Lambda A_{\kappa^{j}}(O))^{sp}}
 \ge (2\kappa\Lambda)^{-sp}\sum_{j\geq m} \frac{\mu( A_{\kappa^{j}}(O))}{\kappa^{jsp}}\\
&\quad \ge \frac{1-c_R}{(2\kappa\Lambda)^{sp}}\sum_{j\geq m} \frac{\mu(B(O,\kappa^j))}{\kappa^{jsp}}
\ge \underbrace{\frac{1-c_R}{c_\sigma(2\kappa\Lambda)^{sp}}\mu(B(O,1))}_{>0}\sum_{j\geq m} 
\kappa^{j(\sigma-sp)}\,.
\end{align*}
Since $\sigma-sp\ge 0$, the last series diverges to infinity, and this is a contradiction
since $E(O)$ is $(s,p)$-thin at infinity.
\end{remark}

We apply the basic technique in Remark \ref{r.pos} again. Namely, we show that, if $sp\le \sigma$
and $E(O)$ is $(s,p)$-thin at infinity, then the existence of a pointwise
limit at infinity along $X\setminus E(O)$ implies that also medians  converge at infinity. The case
$sp>1=\sigma$ of latter Example \ref{e.exe} shows
that the restriction $sp\le \sigma$ is  necessary in this generality.

 \begin{theorem}\label{t.p_implies_m}
Let $0<s\leq 1$ and $0<p<\infty$ be such that
$sp\le \sigma$, where $\sigma$ is as in 
\eqref{iteration reverse}.
Let $u\in L^0(X)$,
$O\in X$ and $E(O)\subset X$ be an $(s,p)$-thin set at infinity such 
that
 \begin{equation}\label{limit convg}
\lim_{\substack{d(x,O)\to \infty \\ x\in X\setminus E(O)}} u(x)=c
 \end{equation}
 for some $c\in\R$. Then 
 \begin{equation*}
     \lim_{j\to\infty}m_u(A_{\kappa^j}(O))=c\,.
 \end{equation*}
\end{theorem}

\begin{proof}
Fix $c\in\R$ as in \eqref{limit convg}. Let $j\in\N$. By Lemma~\ref{med_another lemma}, there exists a set $\tilde{A}_{\kappa^j}(O)\subset {A}_{\kappa^j}(O)$ such that $2\mu(\tilde{A}_{\kappa^j}(O))\geq\mu({A}_{\kappa^j}(O))>0$ and
\begin{equation}\label{e.ee}
    |m_u({A}_{\kappa^j}(O))-c|\leq|u(x)-c|
\end{equation}
whenever $x\in \tilde{A}_{\kappa^j}(O) $.

We claim there exists $M\in\N$ such that
$\tilde{A}_{\kappa^j}(O)\setminus E(O)\not=\emptyset$ for all $j> M$.
We make an antithesis by assuming that $\tilde{A}_{\kappa^j}(O)\subset E(O)$ for infinitely many $j\in\N$.
Denote by ${J}\subset \N$ the infinite set of all such indices $j\in\N$.
Fix $\Lambda>1$ and $m\in\N$. By using monotonicity of relative capacity, \eqref{annuli ball comparison} and \eqref{iteration reverse} we obtain
\begin{align*}
&\sum_{j\ge m} \mathrm{cap}_{s,p}(E(O)\cap A_{\kappa^{j}}(O),\Lambda A_{\kappa^{j}}(O))
\ge \sum_{\substack{j\in{J}\\j\ge m}}\mathrm{cap}_{s,p}(\tilde{A}_{\kappa^j}(O),\Lambda A_{\kappa^{j}}(O))\\
&\quad \ge \sum_{\substack{j\in{J}\\j\ge m}} \frac{||\mathbf{1}_{\tilde{A}_{\kappa^j}(O)}||_{L^p(\Lambda A_{\kappa^{j}}(O))}^p}{\diam(\Lambda A_{\kappa^{j}}(O))^{sp}}
 \ge (2\kappa\Lambda)^{-sp}\sum_{\substack{j\in{J}\\j\ge m}} \frac{\mu(\tilde{A}_{\kappa^j}(O))}{\kappa^{jsp}}
 \ge \frac{1}{2(2\kappa\Lambda)^{sp}}\sum_{\substack{j\in{J}\\j\ge m}} \frac{\mu({A}_{\kappa^j}(O))}{\kappa^{jsp}}\\
&\quad \ge \frac{1-c_R}{2(2\kappa\Lambda)^{sp}}\sum_{\substack{j\in{J}\\j\ge m}} \frac{\mu(B(O,\kappa^j))}{\kappa^{jsp}}
\ge \underbrace{\frac{1-c_R}{2c_\sigma(2\kappa\Lambda)^{sp}}\mu(B(O,1))}_{>0}\sum_{\substack{j\in{J}\\j\ge m}} 
\kappa^{j(\sigma-sp)}\,.
\end{align*}
Since $\sigma-sp\ge 0$ and $\lvert {J}\rvert=\infty$, the last series diverges to infinity, and this is a contradiction
since $E(O)$ is $(s,p)$-thin at infinity.

Let $\varepsilon>0$ be given. Then by \eqref{limit convg} there exists $N\in\N$ such that 
\begin{equation}\label{convergence}
    |u(x)-c|<\varepsilon
\end{equation}
whenever $x\in X\setminus E(O)$ and $d(x,O)>N$. Let $j>\max\{M,\log_{\kappa}N\}$ and $x\in\tilde{A}_{\kappa^j}(O)\setminus E(O)$. By \eqref{e.ee} and \eqref{convergence}, we then have
\begin{equation*}
    |m_u({A}_{\kappa^j}(O))-c|\leq|u(x)-c|<\varepsilon.
\end{equation*}
Hence  $\lim_{j\to\infty}m_u(A_{\kappa^j}(O))=c$.
     \end{proof}

The following lemma is \cite[Lemma 3.1]{KlKoNg}.

\begin{lemma}\label{lm-Klein-original} 
		Let $0< p<\infty$. If $(a_j)_{j\in\N}$ is a non-negative sequence such that $\sum_{j\in\N}a_j<\infty$, then there exists a positive and decreasing sequence $(b_j)_{j\in\N}$ 
	such that  $\lim_{j\to \infty}b_j=0$ and
		\begin{equation*}
		\sum_{j\in\N}  \frac{a_j}{b_j^p} <\infty\,.
		\end{equation*}\end{lemma}

Next we formulate and  prove one of our main results, which is a  
converse of Theorem \ref{t.p_implies_m} for Haj{\l}asz--Sobolev functions.
Namely, the following result gives a sufficient condition for the existence
of their pointwise limit at infinity. This sufficient condition is given in terms 
of converging  medians.

 \begin{theorem}\label{l5}
 Let $0<s\leq1$, $0<p<\infty$, and let $u\in\dot{M}^{s,p}(X)$ be an $M^{s,p}$-quasicontinuous function.
Assume that  $O\in X$ and  the limit
\begin{equation}\label{median limit 2}
\lim_{j\to\infty}m_u(  A_{\kappa^j}(O))=c
\end{equation}
exists.
Then there is a set $E(O)\subset X$, that is $(s,p)$-thin at infinity, such that
 \begin{equation}\label{function limit}
    \lim_{\substack{d(x,O)\to \infty \\ x\in X\setminus E(O)}} u(x)=c\,.
 \end{equation}
  \end{theorem}

\begin{proof}
Let  $g\in\mathcal{D}_s(u)\cap L^p(X)$. 
Fix a constant $\Lambda>1$. Then there exists  $C(\Lambda,\kappa)>0$ such that $\sum_{j\in\N}\mathbf{1}_{\Lambda  A_{\kappa^j}(O)}(x)\leq C(\Lambda,\kappa)$ for all $x\in X$. Hence,  we have 
\begin{align*}
  \sum_{j\in\N}  \int_{\Lambda  A_{\kappa^j}(O)}g^p(x)\, d\mu(x)=\sum_{j\in\N}\int_X\mathbf{1}_{\Lambda  A_{\kappa^j}(O)}(x)\cdot g^p(x)\, d\mu(x)\leq C(\Lambda,\kappa)\int_X g^p(x)\, d\mu(x)<\infty.
\end{align*}
By using Lemma~\ref{lm-Klein-original}, we obtain a positive and decreasing sequence $(b_j)_{j\in\N}$ such that 
\begin{equation}\label{5.1.1}
 \lim_{j\to\infty}b_j=0\qquad\text{ and }\qquad\sum_{j\in\N}\frac{1}{b_j^p}\int_{\Lambda  A_{\kappa^j}(O)}g^p(x)\, d\mu(x)<\infty\,.  
\end{equation}
We define, for every $j\in\N$, 
    \[E_j=\{x\in  A_{\kappa^j}(O)  \,:\, |u(x)-m_u(  A_{\kappa^j}(O))|>b_j\}\,,\]
and  $E(O)=\bigcup_{j\geq 1}E_j$. By Lemma~\ref{l3} and \eqref{5.1.1}, we get
\begin{align*}
   &\limsup_{m\to\infty}\sum_{j\geq m}\mathrm{cap}_{s,p}(E(O)\cap  A_{\kappa^j}(O),\Lambda  A_{\kappa^j}(O))\\
   &\qquad=
   \limsup_{m\to\infty}\sum_{j\geq m} \mathrm{cap}_{s,p}\left(\left(\bigcup_{j\geq 1}E_j\right) \cap   A_{\kappa^j}(O),\Lambda  A_{\kappa^j}(O)\right)\\
   &\qquad= \limsup_{m\to\infty}\sum_{j\geq m} \mathrm{cap}_{s,p}(E_j,\Lambda  A_{\kappa^j}(O) )\\
    &\qquad\leq C(p,c_\mu,c_R,\Lambda) \limsup_{m\to\infty}\sum_{j\geq m}\frac{1}{b_j^p}\int_{\Lambda  A_{\kappa^j}(O)}g^p(x)\, d\mu(x)=0.
\end{align*}
It follows that the set $E(O)$ is $(s,p)$-thin at infinity.

Let $\varepsilon>0$ be given. By using  \eqref{median limit 2} and  \eqref{5.1.1}, we obtain  $M\in\N$ such that $|m_u(A_{\kappa^{n}}(O))-c|<\frac{\varepsilon}{2}$ and $b_{n}<\frac{\varepsilon}{2}$ for all $n\geq M$.
  Choose $N\in\N$ such that  $N\ge \kappa^M$. 
 Fix $x\in X\setminus E(O)$ such that $d(x,O)>N\ge \kappa^M$.   Then there exists $n\geq M$ such that $x\in A_{\kappa^{n}}(O)\setminus E_n$ and 
		\begin{align*}
			|u(x)-c|\leq |u(x)-m_u(A_{\kappa^{n}}(O))|+|m_u(A_{\kappa^{n}}(O))-c|< b_{n}+\frac{\varepsilon}{2}<\varepsilon. 
		\end{align*}
We have shown that 
\eqref{function limit}  holds.
  \end{proof}
  
    \begin{example}\label{e.exe}
  We consider the metric measure space $X=\R$ equipped
  with the Euclidean distance and the one-dimensional Lebesgue measure.
  Then inequality \eqref{reverse doubling} holds for $\kappa=2$.
Let $u\colon \R\to \R$ be any measurable function such
  that $u(x)=-1$ for every $x\le -1$ and $u(x)=1$ for every $x\ge 1$.
Choose $O$ to be the origin, so that
  $ A_{2^j}(O)=(-2^{j+1},2^{j+1})\setminus (-2^j,2^j)$
  for every $j\in\N$.
  By the given properties of $u$, we have
$u_{ A_{2^j}(O)}=1$ for all $j\in\N$, and
  therefore the limit of these medians exists, as $j\to\infty$.
  That is, condition \eqref{median limit 2} holds
  with $c=1$. On the other hand, it is clear
  that $\lim_{x\to\infty} u(x)=1$ and $\lim_{x\to-\infty} u(x)=-1$,
  where the two limits are interpreted in the classical way.
On the other hand, our definition for the limit at infinity in \eqref{function limit} does not distinguish between the two ends of the real axis. Therefore, why does this example not contradict Theorem \ref{l5}? 
The two explanations involving  $0<s\leq1$ and $0<p<\infty$ are as follows:
  \begin{itemize}
\item   If $sp\le 1$, then by
  Lemma \ref{med_int}, applied with the left and right halves
  of $ A_{2^j}(O)$, it is straightforward to show that $u\not\in \dot{M}^{s,p}(\R)$
  independent of its values in  $(-1,1)$.
Therefore Theorem \ref{l5} does not apply to $u$.
  \item If $sp>1$, then by Remark \ref{r.neg} with $Q=1$, the real axis
$\R=E(O)$ is an $(s,p)$-thin set at infinity. In particular, this fits as the claimed exceptional set in Theorem \ref{l5} whether or not $u\in \dot{M}^{s,p}(\R)$.
  \end{itemize}
   \end{example}

 In the following variant of  Theorem \ref{l5}, we 
 focus on  regime $sp<\sigma$. Since we always have  $\sigma\le Q$ if  $Q$ is as in 
\eqref{iteration doubling}, Theorem \ref{t.main} excludes
the possibly vacuous regime $sp>Q$.

   \begin{theorem}\label{t.main}
  Let $0<s\leq1$ and $0<p<\infty$ be such that $sp<\sigma$, where $\sigma>0$ is as in \eqref{iteration reverse}.
Let $u\in\dot{M}^{s,p}(X)$ be an $M^{s,p}$-quasicontinuous function.
Then for every $O\in X$ there is a set $E(O)\subset X$, that is $(s,p)$-thin at infinity, such that
\[
    \lim_{\substack{d(x,O)\to \infty \\ x\in X\setminus E(O)}} u(x)=c\,,
\]
 where the constant $c\in\R$ is as in  Lemma \ref{l4}.
\end{theorem}

\begin{proof}
By Lemma \ref{l4},  there exists a constant $c\in\R$ such that 
$\lim_{j\to\infty}m_u(  A_{\kappa^j}(O))=c$
for all $O\in X$. We emphasize
that the constant $c$ here is independent of $O\in X$. Hence, the claim follows
from Theorem \ref{l5}.
\end{proof}

The following result is a counterpart of Theorem \ref{t.main} for
$u\in M^{s,p}(X)=\dot{M}^{s,p}(X)\cap L^p(X)$
without the restriction $sp<\sigma$. Due to the additional $L^p(X)$-integrability assumption, we
can now omit this parameter restriction and also infer that the limit at infinity is zero.
     
     \begin{theorem}\label{t.main_Lp}
Let $0<s\leq1$,  $0<p<\infty$, and let $u\in {M}^{s,p}(X)$ be an $M^{s,p}$-quasicontinuous function. Then for every $O\in X$ there is a set $E(O)\subset X$, that is $(s,p)$-thin at infinity, such that
\[
\lim_{\substack{d(x,O)\to \infty \\ x\in X\setminus E(O)}} u(x)=0\,.
\]
\end{theorem}

\begin{proof}
Observe  that $u\in L^p(X)$. Hence, by Lemma \ref{median lp cong}, we have
\[
\lim_{j\to\infty}m_u(  A_{\kappa^j}(O))=0
\]
for all $O\in X$. Since we also have $u\in \dot{M}^{s,p}(X)$, the claim follows
from Theorem \ref{l5}.
\end{proof}

\section{Hausdorff content conditions}\label{s.h}

Our  results in Section \ref{s.thin} involve
an exceptional set that is $(s,p)$-thin at infinity. This
notion is defined in terms of a relative capacity. In this section, we show that such sets are thin at infinity in terms of a Hausdorff content of codimension  $sp-\alpha$, for all $0<\alpha<sp$.    See Theorem~\ref{e.hthin}.  This result provides a geometric interpretation
of $(s,p)$-thin sets, and it  allows us to compare
our results favourably to 
some earlier results  involving homogeneous Newtonian functions \cite{KlKoNg}.
We refer to Theorem \ref{msp haus est} and Remark \ref{r.HN} for further details.

\begin{definition}
Let $0<\rho\le\infty$ and $d\ge 0$.
The ($\rho$-restricted) Hausdorff content of codimension $d$ 
of a set $E\subset X$ is defined by  
\begin{equation*}
\begin{split}
    &\mathcal{H}_{\rho}^{\mu,d}(E)\\&=\inf \left\{\sum_{i}\frac{\mu(B(x_i,r_i))}{r_i^d} \,:\,   E\subset \bigcup_{i}B(x_i,r_i)\text{ where $x_i\in X$
    }\text{and  }0<r_i\le \rho\text{ for all $i$}\right\},
    \end{split}
\end{equation*}
where the infimum is taken over all finite and countable coverings.
\end{definition}

Let $E, F\subset X$. Then  it is easy to show that
\begin{align}\label{haus subadditivity}
    \mathcal{H}_{\rho}^{\mu,d}(E\cup F)\leq \mathcal{H}_{\rho}^{\mu,d}(E)+\mathcal{H}_{\rho}^{\mu,d}(F)\,.
\end{align} 
The subadditivity property \eqref{haus subadditivity} 
together with Lemma \ref{full norm cap and haus est} allows us to omit sets of zero capacity when estimating the Hausdoff content with suitable parameters.
The following lemma is known among experts, and the  technique of proof originates already in  \cite{KoHe}, but we recall the proof for convenience.

\begin{lemma}\label{full norm cap and haus est}
Let $0<s\leq1$, $0<p<\infty$ and $0<\alpha<sp$.  Then 
there exists a constant $C=C(s,p,\alpha,c_\mu)>0$ such that  \[
\mathcal{H}_{\infty}^{\mu,sp-\alpha}(E)\leq C\, \mathrm{Cap}_{M^{s,p}}(E)\,.
     \]
     for every  set $E\subset X$.
\end{lemma}

\begin{proof}
Fix a set $E\subset X$.
Let $u$ be a test function for $\mathrm{Cap}_{M^{s,p}}(E)$
as in Definition \ref{d.h_cap}.
That is, we have $u\in M^{s,p}(X)$ and $u\geq 1$ on a neighbourhood $U$ of $E$. Fix $x\in E$. 
Since $x\in U$ and $U$ is open,  there exists $n_x\in\N$ so that  
$B(x,2^{-n_x})\subset U$ and thus $m_u(B(x,2^{-n_x}))\geq1$. Therefore we have
\begin{align}\label{eq1}
    1\leq | m_u(B(x,2^{-n_x}))-m_u(B(x,1))|+|m_u(B(x,1))|
\end{align}
for all $x\in E$.

Define $E_1=\{x\in E\,:\,| m_u(B(x,2^{-n_x}))-m_u(B(x,1))|\geq\frac{1}{2} \}$. 
We now estimate the Hausdorff content of $E_1$ and, for this purpose, we fix $g\in\mathcal{D}_s(u)$.
Fix $x\in E_1$. Denote $r_k=2^{-k}$ and $B(x,r_k)=B_k(x)$ for every $k\in\N_0$.  We have
\begin{align*}
     \frac{1}{2}\leq | m_u(B(x,2^{-n_x}))-m_u(B(x,1))|
&\leq \sum_{k=0}^{n_x-1}|m_u(B_{k+1}(x))-m_u(B_{k}(x))|\\
&\leq\sum_{k=0}^\infty|m_u(B_{k+1}(x))-m_u(B_{k}(x))|.
\end{align*}
 and let $k\in\N_0$.   Then by Lemma~\ref{med_int} we have
\begin{align*}
    |m_u(B_k(x))&-m_u(B_{k+1}(x))|^p\\
 & \leq C(s,p)r_k^{sp} \left(\frac{1}{\mu(B_{k}(x))}\int_{B_{k}(x)}g^p(z)\, d\mu(z)+\frac{1}{ \mu(B_{k+1}(x))}\int_{
   B_{k+1}(x) }g^p(y)\, d\mu(y)\right)\\
   &\leq C(s,p, c_\mu)r_k^{sp}\intav_{B_{k}(x)}g^p(y)\, d\mu(y)\,.
\end{align*}
Let $\varepsilon=\alpha/p>0$ then we have
\[
\sum_{k=0}^\infty 2^{-k\varepsilon}\leq C(\alpha, p)\cdot \frac{1}{2}\leq\sum_{k=0}^\infty  C(\alpha,s,p,c_\mu)r_k^{s}\left(\intav_{B_{k}(x)}g^p(y)\, d\mu(y)\right)^\frac{1}{p}\,.
\]
By comparing the above two series we obtain $k\in\N_0$, depending on $x\in E_1$, so that
\[
r_{k}^\alpha = 2^{-k\varepsilon p}\leq C(s,p,\alpha,c_\mu) r_{k}^{sp}\intav_{B_{k}(x)}g^p(y)\, d\mu(y)\,.
\]
Define  $r_x=r_k$ and $B_x=B_{k}(x)=B(x,r_x)$.
Hence, for every $x\in E_1$, there exists a ball $B_x$ centered at $x$ and
 of radius $r_{x}\le 1$ such that
 \begin{align}\label{haus_rad_grad complete norm_3}
  \mu(B_x) r_{x}^{\alpha-sp} \leq C(s,p,\alpha,c_\mu)\int_{B_{x}}g^p(y)\, d\mu(y)\,.
\end{align}
By the  $5r$-covering lemma \cite[Lemma~1.7]{BB}, we obtain countable number of points $x_j\in E_1$, $j\in J\subset\N$, so that the  balls $B_{x_j}$ are pairwise disjoint and $E_1\subset\bigcup_{j\in J} 5B_{x_j}$. By the doubling condition \eqref{doubling} and  
inequality \eqref{haus_rad_grad complete norm_3}, we then have
\begin{align*}
   \mathcal{H}_{\infty}^{\mu,sp-\alpha}(E_1)
   &\leq\sum_{j\in J}\frac{\mu(5B_{x_j})}{(5r_{x_j})^{sp-\alpha}}\\
    &\leq C(s,p,\alpha,c_\mu)\sum_{j\in J}\int_{B_{x_j}}g^p(y)\, d\mu(y)\\
   &\leq C(s,p,\alpha,c_\mu)\int_{X}g^p(y)\, d\mu(y)\,.
\end{align*}
By taking the  infimum over all $g\in\mathcal{D}_s(u)$, we get
\begin{equation}\label{eq0}
     \mathcal{H}_{\infty}^{\mu,sp-\alpha}(E_1)\leq C(s,p,\alpha,c_\mu)\lVert u\rVert_{\dot{M}^{s,p}(X)}^p\,.
\end{equation}

It remains to estimate the Hausdorff
content of $E\setminus E_1$.
Fix $x\in E\setminus E_1$. By \eqref{eq1}, we have $|m_u(B(x,1))|\geq\frac{1}{2}.$ Then by Lemma~\ref{median prps} we get that
\begin{align}\label{eq3}
   \frac{1}{2}\leq|m_u(B(x,1))|\leq m_{|u|}(B(x,1))&\leq \left(2\intav_{B(x,1)}|u(y)|^p\,  d\mu(y)\right)^{\frac{1}{p}}.
\end{align}
Observe that $E\setminus E_1\subset \bigcup_{x\in E\setminus E_1}B(x,1)$. Then by using the $5r$-covering lemma \cite[Lemma~1.7]{BB}, we obtain points
$x_j\in E\setminus E_1$, $j\in J\subset\N$, such that $E\setminus E_1\subset \bigcup_{j\in J}5B_{x_j}$, the balls $B_{x_j}=B(x_j,r_{x_j})$ are pairwise disjoint and $r_{x_j}=1$ for all $j\in J$. By the doubling condition \eqref{doubling} and inequality \eqref{eq3}, we obtain
\begin{equation}\label{eq4}
\begin{split}
 \mathcal{H}_{\infty}^{\mu,sp-\alpha}(E\setminus E_1)&\leq\sum_{j\in J}\frac{\mu(5B_{x_j})}{(5r_{x_j})^{sp-\alpha}}
 \\
  &\leq C(p,c_\mu)\sum_{j\in J} \int_{B_{x_j}}|u(y)|^p\,  d\mu(y) \leq C(p,c_\mu)\lVert u\rVert_{L^p(X)}^p.
  \end{split}
\end{equation}

Finally, by combining inequalities \eqref{haus subadditivity}, \eqref{eq0} and \eqref{eq4} we get that 
\begin{align*}
\mathcal{H}_{\infty}^{\mu,sp-\alpha}(E)&\le \mathcal{H}_{\infty}^{\mu,sp-\alpha}(E_1) + \mathcal{H}_{\infty}^{\mu,sp-\alpha}(E\setminus E_1) \\&\leq C(s,p,\alpha,c_\mu)\lVert u\rVert_{\dot{M}^{s,p}(X)}^p+C(p,c_\mu)\lVert u\rVert_{L^p(X)}^p\le C(s,p,\alpha,c_\mu)\lVert u\rVert_{M^{s,p}(X)}^p\,.
    \end{align*}
Hence, we get the desired inequality 
$\mathcal{H}_{\infty}^{\mu,sp-\alpha}(E)\leq C\, \mathrm{Cap}_{M^{s,p}}(E)$
by taking the infimum over all test functions $u$  for $\mathrm{Cap}_{M^{s,p}}(E)$.
\end{proof}

The following
lemma is a  key tool when showing that an
$(s,p)$-thin set at infinity is
thin at infinity also in terms of Hausdorff contents
of codimension $sp-\alpha$, for all $0<\alpha<sp$.

\begin{lemma}\label{homogeneous Cap and Hauss est}
Let $0<s\leq1$, $0<p<\infty$, $\Lambda>1$ and $0<\alpha<sp$. Fix $O\in X$ and $j\in\N$.  Then for every $E\subset  A_{\kappa^j}(O)$ we have
\[
\frac{\mathcal{H}_{\infty}^{\mu,sp-\alpha}(E)}{\kappa^{j\alpha}}\leq C\, \mathrm{cap}_{s,p}(E,\Lambda  A_{\kappa^j}(O))
\]
     for a constant $C=C(s,p,\alpha,c_\mu,\kappa,\Lambda)$.
\end{lemma}

\begin{proof}
  Fix a set $E\subset  A_{\kappa^j}(O)$ and denote
 $r_0=(1-\Lambda^{-1})\kappa^j$.
Observe that  $r_0\le (\Lambda -1)\kappa^{j+1}$ and therefore
 $B(x,r_0)\subset \Lambda  A_{\kappa^j}(O)$
 for every $x\in E$.  
  Let $v$ be a test function for $\text{cap}_{s,p}(E,\Lambda  A_{\kappa^j}(O))$ as in Definition \ref{d.relative}. 
  Thus $v\in\dot{M}^{s,p}(X)$ is $M^{s,p}$-quasicontinuous and satisfies $v\ge 1$ in $E$.

We will consider two cases and combine them at the end of the proof.
First we assume that there exists $x\in E$ such that
\begin{equation}\label{e.first_case}
\lvert m_v(B(x,r_0))\rvert\ge \frac{1}{2}\,.
\end{equation}
Then
by Lemma~\ref{median prps} we get that
\[
\frac{1}{2}\le \lvert m_v(B(x,r_0))\rvert\le m_{\lvert v\rvert}(B(x,r_0))\le \left(2\intav_{B(x,r_0)}|v(y)|^p\,  d\mu(y)\right)^{\frac{1}{p}}\,.
\]
By using also inequality \eqref{iteration doubling}, we get
\begin{align*}
\mu(B(O,\kappa^{j+1}))&\le C(c_\mu,\kappa,\Lambda)\mu(B(x,r_0))\\&\le C(p,c_\mu,\kappa,\Lambda)\int_{B(x,r_0)}|v(y)|^p\,  d\mu(y)\le C(p,c_\mu,\kappa,\Lambda)\int_{\Lambda  A_{\kappa^j}(O)}|v(y)|^p\,  d\mu(y)\,.
\end{align*}
Since $E\subset B(O,\kappa^{j+1})$, we can estimate
\begin{equation}\label{e.first}
\begin{split}
\frac{\mathcal{H}_{\infty}^{\mu,sp-\alpha}(E)}{\kappa^{j\alpha}}
&\le \frac{1}{\kappa^{j\alpha}}\cdot \frac{\mu(B(O,\kappa^{j+1}))}{(\kappa^{j+1})^{sp-\alpha}}\le \frac{C(p,c_\mu,\kappa,\Lambda)}{\kappa^{jsp}}\int_{\Lambda  A_{\kappa^j}(O)}|v(y)|^p\,  d\mu(y)
\\&\le \frac{C(s,p,c_\mu,\kappa,\Lambda)}{\diam(\Lambda  A_{\kappa^j}(O))^{sp}}\int_{\Lambda  A_{\kappa^j}(O)}|v(y)|^p\,  d\mu(y)\,.
\end{split}
\end{equation}

We now turn to the second case of the proof by assuming that, 
 for all $x\in E$, 
\begin{equation}\label{e.case2}
\lvert m_v(B(x,r_0))\rvert< \frac{1}{2}\,.
\end{equation}
Let $v^*(x)=\limsup_{r\to 0} m_v(B(x,r))$ for all $x\in X$.
We denote 
\[
E_1=\{x\in E \,:\, v(x)=v^*(x)\}\quad\text{ and }\quad E_2=\{x\in E\,:\, \text{limit }\lim_{r\to 0} m_v(B(x,r))\text{ exists}\}\,.
\]
We also denote $F_1=E\setminus E_1$ and $F_2=E\setminus E_2$.
 Theorem \ref{prop0-1} and Theorem \ref{quasicont repesent} 
combined show that $\mathrm{Cap}_{M^{s,p}}(F_1)=0$ and $\mathrm{Cap}_{M^{s,p}}(F_2)=0$.
Observe that $E_1\cap E_2=E\setminus (F_1\cup F_2)$. 
Inequality \eqref{haus subadditivity} and Lemma \ref{full norm cap and haus est} imply that
\begin{equation}\label{e.ex}
\begin{split}
\mathcal{H}_{\infty}^{\mu,sp-\alpha}(E)
&=\mathcal{H}_{\infty}^{\mu,sp-\alpha}(E_1\cap E_2)+\mathcal{H}_{\infty}^{\mu,sp-\alpha}(F_1\cup F_2)
\\&\le \mathcal{H}_{\infty}^{\mu,sp-\alpha}(E_1\cap E_2)+\mathcal{H}_{\infty}^{\mu,sp-\alpha}(F_1)+\mathcal{H}_{\infty}^{\mu,sp-\alpha}(F_2)=\mathcal{H}_{\infty}^{\mu,sp-\alpha}(E_1\cap E_2)\,.
\end{split}
\end{equation}
Hence, it suffices to estimate $\mathcal{H}_{\infty}^{\mu,sp-\alpha}(E_1\cap E_2)$. 

Fix  $g\in\mathcal{D}^s(v)$.
Let 
 $x\in E_1\cap E_2\subset E$. 
 Define   $r_k=\kappa^{-k}r_0$   and $B_k(x)=B(x,r_k)$ for all $k\in\N_0$.
We observe by using the case inequality \eqref{e.case2} that
\begin{align*}
1\le v(x)=\lvert v(x)\rvert &\le \lvert v(x)-m_v(B(x,r_0))\rvert+\lvert m_v(B(x,r_0))\rvert\\&< \lvert v(x)-m_v(B(x,r_0))\rvert + \frac{1}{2}\,.
\end{align*}
Hence, using also the assumption that $x\in E_1\cap E_2$, we obtain
\begin{equation}\label{e.start}
\begin{split}
\frac{1}{2}\le \lvert v(x)-m_v(B(x,r_0))\rvert
&=\lim_{n\to \infty} \lvert m_v(B(x,r_n))-m_v(B(x,r_0))\rvert\\
&\le \sum_{k=0}^{\infty} \lvert m_v(B(x,r_{k+1}))-m_v(B(x,r_k))\rvert\,.
\end{split}
\end{equation}
Fix $k\in\N_0$.
 Then by Lemma~\ref{med_int} we have 
 \begin{equation}\label{e.sec}
\begin{split}
    |m_v&(B_{k+1}(x))-m_v(B_k(x))|^p\\ 
  &\leq C(s,p)r_k^{sp} \left(\frac{1}{\mu(B_{k+1}(x))}\int_{B_{k+1}(x)}g^p(z)\, d\mu(z)+\frac{1}{ \mu(B_{k}(x))}\int_{
   B_{k}(x) }g^p(y)\, d\mu(y)\right)\\
   &\leq C(s,p,\kappa,c_\mu)r_k^{sp}\intav_{B_{k}(x)}g^p(y)\, d\mu(y)\,.
   \end{split}
   \end{equation}
Let $\varepsilon=\alpha/p>0$. By combining  the estimates \eqref{e.start} and \eqref{e.sec}, we get
\[
\sum_{k=0}^\infty \kappa^{-k\varepsilon}= C(\alpha, p,\kappa)\cdot \frac{1}{2}\leq\sum_{k=0}^\infty  C(s,p,\alpha,\kappa,c_\mu)r_k^{s}\left(\intav_{B_{k}(x)}g^p(y)\, d\mu(y)\right)^\frac{1}{p}\,.
\]
By comparing the above two series we obtain $k\in\N_0$, depending on $x\in E_1\cap E_2$, so that
\[
\frac{r_{k}^\alpha}{r_0^\alpha}=\kappa^{-k\alpha}= \kappa^{-k\varepsilon p}\leq C(s,p,\alpha,\kappa,c_\mu) r_{k}^{sp}\intav_{B_{k}(x)}g^p(y)\, d\mu(y)\,.
\]
We denote $r_x=r_k\le r_0$ and $B_x=B_{k}(x)=B(x,r_x)$. Then, for every $x\in E_1\cap E_2$, there exists a ball $B_x\subset B(x,r_0)\subset \Lambda  A_{\kappa^j}(O)$ centered at $x$ and
 of radius $r_{x}\le r_0$ such that
 \begin{align}\label{haus_rad_grad complete norm_2}
\frac{1}{r_0^\alpha}\cdot \frac{\mu(B_x)}{r_{x}^{sp-\alpha}} \leq C(s,p,\alpha,\kappa,c_\mu)\int_{B_{x}}g^p(y)\, d\mu(y)\,.
\end{align}
By the  $5r$-covering lemma \cite[Lemma~1.7]{BB}, we obtain  points $x_m\in E_1\cap E_2$, $m\in M\subset\N$, so that the  balls $B_{x_m}$ are pairwise disjoint and $E_1\cap E_2\subset\bigcup_{m\in M} 5B_{x_m}$. 

Recall that
$r_0=(1-\Lambda^{-1})\kappa^j$
and $B_{x_m}\subset \Lambda  A_{\kappa^j}(O)$ for all $m$. By inequalities \eqref{doubling} and  
 \eqref{haus_rad_grad complete norm_2}, we then have
\begin{align*}
   \frac{\mathcal{H}_{\infty}^{\mu,sp-\alpha}(E_1\cap E_2)}{\kappa^{j\alpha}}&\le C(\alpha,\Lambda) \frac{\mathcal{H}_{\infty}^{\mu,sp-\alpha}(E_1\cap E_2)}{r_0^\alpha}\\
   &\leq \frac{C(\alpha,\Lambda)}{r_0^\alpha}\sum_{m\in M}\frac{\mu(5B_{x_m})}{(5r_{x_m})^{sp-\alpha}}\\
    &\leq C(s,p,\alpha,\kappa,c_\mu,\Lambda)\sum_{m\in M}\int_{B_{x_m}}g^p(y)\, d\mu(y)\\
   &\leq C(s,p,\alpha,\kappa,c_\mu,\Lambda)\int_{\Lambda  A_{\kappa^j}(O)}g^p(y)\, d\mu(y)\,.
\end{align*}
By taking also \eqref{e.ex} into account and taking the  infimum over all $g\in\mathcal{D}_s(v)$, we get
\begin{equation}\label{eq7}
     \frac{\mathcal{H}_{\infty}^{\mu,sp-\alpha}(E)}{\kappa^{j\alpha}}\le  \frac{\mathcal{H}_{\infty}^{\mu,sp-\alpha}(E_1\cap E_2)}{\kappa^{j\alpha}}\leq C(s,p,\alpha,\kappa,c_\mu,\Lambda)
     \inf_{g\in \mathcal{D}_s(v)}\lVert g\rVert_{L^p(\Lambda  A_{\kappa^j}(O))}^p\,.
\end{equation} 
 
Depending
whether \eqref{e.first_case} or \eqref{e.case2} holds, 
we can use the corresponding inequality \eqref{e.first} or \eqref{eq7}, respectively, to deduce that
\[
\frac{\mathcal{H}_{\infty}^{\mu,sp-\alpha}(E)}{\kappa^{j\alpha}}
\le C(s,p,\alpha,\kappa,c_\mu,\Lambda)\left(\frac{\lVert v\rVert_{L^p(\Lambda  A_{\kappa^j}(O))}^p}{\diam(\Lambda  A_{\kappa^j}(O))^{sp}}+\inf_{g\in \mathcal{D}_s(v)}\lVert g\rVert_{L^p(\Lambda  A_{\kappa^j}(O))}^p\right)\,.
\]
The claim follows by taking infimum over test functions $v$ for $\text{cap}_{s,p}(E,\Lambda  A_{\kappa^j}(O))$.
\end{proof}

Next we  provide some geometric properties
of $(s,p)$-thin sets.

\begin{theorem}\label{e.hthin}
Let $0<s\leq 1$ and $0<p<\infty$ be such that
$sp\le \sigma$, where $\sigma>0$ is as in \eqref{iteration reverse}.
Assume that $O\in X$ and that a set $E(O)\subset X$ is $(s,p)$-thin at infinity.
Then
\begin{equation*}
    \lim_{m\to\infty}\sum_{j\geq m}\frac{\mathcal{H}_{\kappa^{j+1}}^{\mu,sp-\alpha}(E(O)\cap  A_{\kappa^j}(O))}{\kappa^{j\alpha}}=0\qquad \text{ for all }0<\alpha < sp\,.
\end{equation*}
\end{theorem}

  \begin{proof}
Fix  $0<\alpha<sp$ and $\Lambda>1$. 
Since $0<sp-\alpha<sp\le \sigma$, by using \cite[Theorem 5.3]{MR4866613}
with $q=sp-\alpha$ and $c_q=c_\sigma$, we obtain
\[
\mathcal{H}_{\kappa^{j+1}}^{\mu,sp-\alpha}(E(O)\cap  A_{\kappa^j}(O))
\le C(c_\sigma,c_\mu)\mathcal{H}_{\infty}^{\mu,sp-\alpha}(E(O)\cap  A_{\kappa^j}(O))
\]
for every $j\in\N$.
On the other hand, by Lemma~\ref{homogeneous Cap and Hauss est}  and Definition~\ref{sp thiness}, we have
\begin{equation*}\begin{split}
    \limsup_{m\to\infty}\sum_{j\geq m} &\frac{\mathcal{H}_{\infty}^{\mu,sp-\alpha}(E(O)\cap  A_{\kappa^j}(O))}{\kappa^{j\alpha}}\\&\leq C(s,p,\alpha,c_\mu,\kappa,\Lambda)\lim_{m\to\infty}\sum_{j\geq m} \mathrm{cap}_{s,p}(E(O)\cap   A_{\kappa^j}(O),\Lambda  A_{\kappa^j}(O))=0\,.
\end{split}\end{equation*}
Hence, the claim follows by combining the above estimates.
  \end{proof}

The following result is a variant
of Theorem \ref{t.main}, where the exceptional
set $E(O)\subset X$ is shown to be thin at infinity with respect to a Hausdorff content of codimension $sp-\alpha$.
There is also a corresponding counterpart of Theorem \ref{t.main_Lp}, whose formulation we omit.

\begin{theorem}\label{msp haus est}
Let $0<s\leq 1$ and $0<p<\infty$ be such that
$sp<\sigma$, where $\sigma>0$ is as in \eqref{iteration reverse}.
Let $u\in\dot{M}^{s,p}(X)$ be an $M^{s,p}$-quasicontinuous function.
Then for every $O\in X$ there is a set $E(O)\subset X$ such that
\begin{equation}\label{e.h_thin}
    \lim_{m\to\infty}\sum_{j\geq m}\frac{\mathcal{H}_{\kappa^{j+1}}^{\mu,sp-\alpha}(E(O)\cap  A_{\kappa^j}(O))}{\kappa^{j\alpha}}=0\qquad \text{ for all }0<\alpha < sp
\end{equation}
and
 \begin{equation}\label{function limit_main_4}
    \lim_{\substack{d(x,O)\to \infty \\ x\in X\setminus E(O)}} u(x)=c\,,
 \end{equation}
 where $c\in\R$ is as in  Lemma \ref{l4}.
\end{theorem}

  \begin{proof}
  Fix $O\in X$.
By  Theorem~\ref{t.main} we obtain a set $E(O)\subset X$ which is $(s,p)$-thin at infinity and satisfies  \eqref{function limit_main_4}. Then by Theorem~\ref{e.hthin}, we obtain 
\[
    \lim_{m\to\infty}\sum_{j\geq m}\frac{\mathcal{H}_{\kappa^{j+1}}^{\mu,sp-\alpha}(E(O)\cap  A_{\kappa^j}(O))}{\kappa^{j\alpha}}=0\qquad \text{ for all }0<\alpha < sp\,.
    \]
Hence, the claim follows.     
  \end{proof}

  \begin{remark}\label{r.HN}
This remark applies Poincar\'e inequalities and upper gradients, 
and we refer to  \cite{BB,MR3363168} for their definitions.
    Let $1<p<\sigma$, where $\sigma>0$ is defined as in \eqref{iteration reverse}. 
    Assume that $X$ supports a
  $q$-Poincar\'e inequality for some $1\le q<p<\infty$.
    Denote
  by $\dot{N}^{1,p}(X)\subset L^0(X)$ the  space of functions that are integrable on balls and have
  a $p$-integrable upper gradient $g$ in $X$.
 This vector space is equipped with the seminorm
  \[
  \lVert u\rVert_{\dot{N}^{1,p}(X)}=\inf_g\lVert g\rVert_{L^p(X)}\,,
\]
where the infimum is taken over all upper gradients $g$ of $u\in \dot{N}^{1,p}(X)$.
Under these assumptions $\dot{N}^{1,p}(X)\subset\dot{M}^{1,p}(X)$ and the embedding is bounded; these facts  follow
from the usual chaining argument and boundedness of maximal operator
on $L^{p/q}(X)$, we refer to \cite[Theorem 9.4]{MR2039955}. Hence, 
our Theorem~\ref{msp haus est} applies to
 functions in homogeneous Newton--Sobolev space $\dot{N}^{1,p}(X)$.
Consequently, this theorem generalizes \cite[Theorem 1.1]{KlKoNg}, which
studies
limits at infinity for functions in $\dot{N}^{1,p}(X)$.
The basic setting in \cite{KlKoNg} is comprised of an unbounded complete metric space $X$ equipped with an Ahlfors $d$-regular measure $\mu$ supporting a $p$-Poincar\'e inequality with $1\le p < d$. In this setting $\sigma=d=Q$, $X$ is connected and a $q$-Poincar\'e
inequality holds for some $1\le q<p$ by the Keith--Zhong theorem \cite{MR2415381,MR3363168}.
In this setting, the  `$p$-thin set at infinity' condition in \cite{KlKoNg}  is essentially  the same as \eqref{e.h_thin} with $s=1$ and $\kappa =2$. Note that
$X$ is connected by the assumed Poincar\'e inequality, and therefore we can choose $\kappa=2$ in \eqref{reverse doubling}.
 \end{remark}

\section{Application to fractional Sobolev spaces}\label{s.fract_app}

Our main results Theorem \ref{t.main}, Theorem \ref{t.main_Lp} and Theorem \ref{msp haus est} are concerned with limits of Haj{\l}asz--Sobolev functions at infinity.
We    initially   chose to work with a homogeneous Haj{\l}asz--Sobolev  space $\dot{M}^{s,p}(X)$ for two reasons. First, this space is relatively easy to work with due to a simple definition of $s$-gradients. Second,
this space is robust since several other function spaces admit a bounded embedding into a homogeneous Haj{\l}asz--Sobolev space.   The reader may consider
Remark~\ref{r.HN} for an example. For another example,  
assuming that  $0<p,q\le \infty$ and $0<s\le 1$, then  so-called Haj{\l}asz--Triebel--Lizorkin
space $\dot{M}^{s,p}_q(X)$ embeds
boundedly  to the homogeneous  Haj{\l}asz--Sobolev space $\dot{M}^{s,p}(X)$, see \cite[Proposition 2.1]{MR2764899}.

In this section, we use our
results to study limits at infinity for
fractional Sobolev functions.  This complements and extends the recent results in \cite{AKM}
that apply in Euclidean spaces.

\begin{definition}
Let $0<s\leq1$ and $0<p,q<\infty$. The homogeneous fractional Sobolev space $\dot{W}_q^{s,p}(X)$ consists of all  $u\in L^0(X)$ for which  
	\[
	\|u\|_{\dot{W}_q^{s,p}(X)} = \left(\int_{X}\left(\int_{X}\frac{|u(x)-u(y)|^q}{d(x,y)^{sq}\mu(B(x,d(x,y)))}\,d\mu(y)\right)^{\frac{p}{q}}\,d\mu(x)\right)^\frac{1}{p}<\infty\,.
	\]
The Fractional Sobolev space  $W^{s,p}_q(X)$ is $\dot{W}^{s,p}_q(X)\cap L^p(X)$ equipped with the quasi-seminorm 
\[
     \|u\|_{W^{s,p}_q(X)}=\|u\|_{L^p(X)}+\|u\|_{\dot{W}_q^{s,p}(X)}\,.
\]

If $q=p$, then $W^{s,p}_q(X)=W^{s,p}_p(X)$ is a generalization of the well known Gagliardo--Sobolev space or a fractional Sobolev space from the  Euclidean setting. We refer to 
\cite{MR2944369,MR4480576,MR4567945,MR2250142} for expositions on these  spaces
that have attracted plenty of interest in recent years.

 The following embedding lemma might be of independent interest, since it covers
the full range $p,q\in (0,\infty)$.
When $p,q\in (1,\infty)$ this lemma is a  consequence
of  \cite[Lemma 6.5]{MR4894884}. We also refer to \cite[Proposition 4.1]{MR2764899}.

\begin{lemma}\label{W,s,p,q-M,s,p}
 Let $0<s\leq1$ and $0<p,q<\infty$. Then 
 \[ 
 \dot{W}^{s,p}_q(X)\subset\dot{M}^{s,p}(X)
 \]
    and
    \begin{equation*}
    \|u\|_{\dot{M}^{s,p}(X)}\leq C(q,c_{\mu}) \|u\|_{\dot{W}^{s,p}_q(X)}
    \end{equation*}
    for all $u\in \dot{W}^{s,p}_q(X).$
\end{lemma}

\begin{proof}
 Let $u\in \dot{W}^{s,p}_q(X)$. Define $g\colon X\to[0,\infty]$ as
  \begin{equation*}
      g(x)=\left(\int_{X}\frac{|u(x)-u(z)|^q}{d(x,z)^{sq}\mu(B(x,d(x,z)))}\,d\mu(z)\right)^\frac{1}{q}\,,\qquad x\in X\,.
  \end{equation*}
  Observe that \[\lVert g\rVert_{L^p(X)}=\lVert u\rVert_{\dot{W}^{s,p}_q(X)}<\infty\,.\]
Hence, it suffices to show that a constant multiple of $g$ is an $s$-gradient of $u$.
  
  Since $u\in L^0(X)$, there exists a set $E\subset X$ such that
  $\mu(E)=0$ and $u$ is finite in $X\setminus E$.
  Fix $x,y\in X\setminus E$ such that $x\not=y$. Then 
  \begin{equation}\begin{split}\label{W,s,p cont M,s,p eq0}
      |u(x)-u(y)|&\leq |u(x)-m_u(B(x,d(x,y)))|+|u(y)-m_u(B(y,d(x,y)))|\\
      &\hspace{35mm}+|m_u(B(x,d(x,y)))-m_u(B(y,d(x,y)))|\,.
  \end{split}\end{equation}
 Using Lemma~\ref{med_another lemma}, 
 with $E=B(x,d(x,y))$ and the constant $c=u(x)\in\R$,
 we obtain a set
$\tilde{B}_{x,y}\subset B(x,d(x,y))$ such that $2\mu(\tilde{B}_{x,y})\geq \mu({B}(x,d(x,y)))$ and 
\[
|u(x)-m_u(B(x,d(x,y)))|=|m_u(B(x,d(x,y)))-c|\le |u(z)-c|=|u(x)-u(z)|
\] for all $z\in\tilde{B}_{x,y}$. Therefore
\begin{align*}
|u(x)-m_u(B(x,d(x,y)))|&\leq\left(\int_{\tilde{B}_{x,y}}\frac{|u(x)-u(z)|^q}{\mu(\tilde{B}_{x,y})}\,d\mu(z)\right)^\frac{1}{q}\\
&\leq 2^\frac{1}{q}\left(\int_{{B}(x,d(x,y))}\frac{|u(x)-u(z)|^q}{\mu({B}(x,d(x,y)))}\,d\mu(z)\right)^\frac{1}{q}\\
&\leq 2^\frac{1}{q}d(x,y)^s\left(\int_{{B}(x,d(x,y))}\frac{|u(x)-u(z)|^q}{d(x,z)^{sq}\mu({B}(x,d(x,z)))}\,d\mu(z)\right)^\frac{1}{q}\,,
\end{align*}
where the last step follows since $d(x,z)< d(x,y)$ for all $z\in {B}(x,d(x,y))$. Thus,
\begin{equation}\label{W,s,p cont M,s,p eq1}
    |u(x)-m_u(B(x,d(x,y)))|\leq 2^\frac{1}{q}d(x,y)^sg(x).
\end{equation}
In a similar way, by interchanging $x$ and $y$, we obtain
\begin{equation}\label{W,s,p cont M,s,p eq2}
    |u(y)-m_u(B(y,d(x,y)))|\leq 2^\frac{1}{q}d(x,y)^sg(y).
\end{equation}

For the remaining term in \eqref{W,s,p cont M,s,p eq0}, we have by Lemma~\ref{median prps} and triangle inequality, 
\begin{equation}\label{W,s,p cont M,s,p eq3}
\begin{split}
     |m_u&(B(x,d(x,y)))-m_u(B(y,d(x,y)))|\\
     & = |m_{u-m_u(B(y,d(x,y)))}(B(x,d(x,y)))|\\
     & \le m_{\lvert u-m_u(B(y,d(x,y)))\rvert}(B(x,d(x,y)))\\
     &\leq \left(2\intav_{B(x,d(x,y))}|u(z)-m_u(B(y,d(x,y)))|^q\, d\mu(z)\right)^\frac{1}{q}\\
     &\leq C(q)\left(\intav_{B(x,d(x,y))}|u(z)-u(y)|^q\,d\mu(z)\right.\\
     &\left.\qquad\qquad+\intav_{B(x,d(x,y))}|u(y)-m_u(B(y,d(x,y)))|^q\, d\mu(z)\right)^\frac{1}{q}.
\end{split}
\end{equation}
Using \eqref{W,s,p cont M,s,p eq3} and \eqref{W,s,p cont M,s,p eq2}, followed triangle inequality and
\eqref{doubling}, we get
\begin{equation}\label{W,s,p,q, M,s,p,q eq000}
    \begin{split}
        |m_u&(B(x,d(x,y)))-m_u(B(y,d(x,y)))|\\
         &\leq C(q)\left(\int_{B(x,d(x,y))}\frac{|u(z)-u(y)|^q}{\mu(B(x,d(x,y)))}\,d\mu(z)+|u(y)-m_u(B(y,d(x,y)))|^q\, \right)^\frac{1}{q}\\
          &\leq C(q)\left(\int_{B(y,2d(x,y))}\frac{|u(z)-u(y)|^q}{\mu(B(x,d(x,y)))}\,d\mu(z)+2(d(x,y)^sg(y))^q\, \right)^\frac{1}{q}\\
          &\leq C(q)\left((2d(x,y))^{sq} c_\mu^2\int_{B(y,2d(x,y))}\frac{|u(y)-u(z)|^q}{d(y,z)^{sq}\mu(B(y,d(y,z)))}\,d\mu(z)+2(d(x,y)^sg(y))^q\, \right)^\frac{1}{q}\\
           &\leq C(q,c_{\mu})\left(d(x,y)^{sq}g^q(y)+d(x,y)^{sq}g^q(y) \right)^\frac{1}{q}
           \leq C(q,c_{\mu})d(x,y)^sg(y)\,.
          \end{split}
\end{equation}
Combining estimates \eqref{W,s,p cont M,s,p eq0}, \eqref{W,s,p cont M,s,p eq1}, \eqref{W,s,p cont M,s,p eq2} and \eqref{W,s,p,q, M,s,p,q eq000}, we get
\begin{equation*}
    |u(x)-u(y)|\leq C(q,c_{\mu})d(x,y)^s(g(x)+g(y))
\end{equation*}
for every $x,y\in X\setminus E$. Define  $h=C(q,c_{\mu})g$. Then $h\in \mathcal{D}_s(u)$ and since
$u\in \dot{W}^{s,p}_q(X)$, we get
\begin{align*}
    \|u\|_{\dot{M}^{s,p}(X)}&\leq \|h\|_{L^p(X)}= C(q,c_{\mu})\lVert u\rVert_{\dot{W}^{s,p}_q(X)}<\infty\,.
    \end{align*}
 We can conclude that $u\in\dot{M}^{s,p}(X)$.
\end{proof}

We  need the notion of quasicontinuity in the fractional Sobolev setting.
For this purpose, we first define the following fractional Sobolev capacity.
The reader is encouraged to recall Definition~\ref{d.h_cap} for the corresponding definition of a  Haj{\l}asz capacity.

      \begin{definition}
Let $0<s\leq1$ and $0<p,q<\infty$. 
The $W^{s,p}_q$-capacity of a set $E\subset X$ is 
\begin{equation*}
   \mathrm{Cap}_{W^{s,p}_q}(E)=\inf\left\{\|u\|_{W^{s,p}_q(X)}^p\,:\, u\in W^{s,p}_q(X)  \text{ and } u\geq1 \text{ on a neighbourhood of $E$}\right\}.
\end{equation*}
If there are no test functions that satisfy the above conditions, we set $\mathrm{Cap}_{W^{s,p}_q}(E)=\infty$.\end{definition}
\end{definition}

Lemma \ref{W,s,p,q-M,s,p}  implies the following
comparison of capacities. See  \cite{MR4894884,MR4870861} for similar results.

\begin{lemma}\label{cap comparison}
Let $0<s\leq1$ and $0<p,q<\infty$. Then there exists a constant $C=C(p,q, c_\mu)$ such that, for every $E\subset X$, we have
    \begin{equation}
        \mathrm{Cap}_{M^{s,p}}(E)\leq C\,\mathrm{Cap}_{W^{s,p}_q}(E)\,.
    \end{equation}
     \end{lemma}
    
    \begin{proof}
    Fix a set $E\subset X$ such that  $\mathrm{Cap}_{W^{s,p}_q}(E)<\infty$. Let 
    $u$ be a test function for $\mathrm{Cap}_{W^{s,p}_q}(E)$. Then
    $u\in W^{s,p}_q(X)$ and $u\geq 1$ on a neighbourhood of $E$.
    Lemma \ref{W,s,p,q-M,s,p} implies that $u\in M^{s,p}(X)$  and
    \begin{align*}
    \mathrm{Cap}_{M^{s,p}}(E) \le \lVert u\rVert_{M^{s,p}(X)}^p
    &=\left(\lVert u\rVert_{L^p(X)} + \lVert u\rVert_{\dot{M}^{s,p}(X)}\right)^p
    \\&\le \left(\lVert u\rVert_{L^p(X)} + C(q,c_\mu)\lVert u\rVert_{\dot{W}^{s,p}(X)}\right)^p\le C(p,q,c_\mu)\lVert u\rVert_{W^{s,p}(X)}^p\,.
\end{align*}
The claim follows by taking infimum over all test functions $u$ as above.
    \end{proof}
    
    The following definition of quasicontinuity is a counterpart of  Definition~\ref{d.qc}.

\begin{definition}
    A function $u\in L^0(X)$ is  $W^{s,p}_q$-quasicontinuous in $X$, if for all $\varepsilon>0$ there exists a set $E\subset X$ such that $\mathrm{Cap}_{W^{s,p}_q}(E)<\varepsilon$ and the restriction of $u$ to $X\setminus E$ is finite and continuous.
\end{definition}
   
The following theorem is concerned with 
the limit at infinity for functions in homogeneous fractional Sobolev spaces.
We refer to \cite{AKM} for similar results in Euclidean spaces.

\begin{theorem}\label{W,s,p,q 1}
Let $0<s\leq 1$ and $0<p,q<\infty$ be such that
$sp<\sigma$, where $\sigma$ is as in \eqref{iteration reverse}.
Let $u\in\dot{W}^{s,p}_q(X)$ be an $W^{s,p}_q$-quasicontinuous function.
Then for every $O\in X$ there is a set $E(O)\subset X$ such that
\begin{equation*}
    \lim_{m\to\infty}\sum_{j\geq m}\frac{\mathcal{H}_{\kappa^{j+1}}^{\mu,sp-\alpha}(E(O)\cap  A_{\kappa^j}(O))}{\kappa^{j\alpha}}=0\qquad \text{ for all }0<\alpha < sp
\end{equation*}
and
\[
    \lim_{\substack{d(x,O)\to \infty \\ x\in X\setminus E(O)}} u(x)=c\,,
    \]
 where $c\in\R$ is as in  Lemma \ref{l4}.
\end{theorem}

\begin{proof}
Let $C=C(p,q,c_\mu)>0$  be as in Lemma \ref{cap comparison}.
Let $u\in \dot{W}^{s,p}_q(X)$ be a $W^{s,p}_q$-quasicontinuous function. By Lemma~\ref{W,s,p,q-M,s,p}, we have $u\in \dot{M}^{s,p}(X)$. 
Fix $\varepsilon>0$.
Since $u$ is  $W^{s,p}_q$-quasicontinuous, there exists a set $E\subset X$ such that $\mathrm{Cap}_{W^{s,p}_q}(E)<C^{-1}\varepsilon$ and the restriction of $u$ to  $X\setminus E$ is continuous. Lemma~\ref{cap comparison} implies that
$\mathrm{Cap}_{M^{s,p}}(E)\le C\,\mathrm{Cap}_{W^{s,p}_q}(E)  < \varepsilon$, and therefore
 $u$ is an $M^{s,p}$ quasicontinuous in $X$. Hence, the claim follows by applying Corollary~\ref{msp haus est}.
\end{proof}

\end{document}